\documentclass[11pt]{amsart}    %SP√ÑTER ZUR√úCK ZU 11PT
\usepackage{latexsym}
\usepackage{amsfonts}
\usepackage{amsmath,amsthm,amssymb}
\usepackage{fullpage}   %SP√ÑTER WIEDER EINF√úGEN
\usepackage{mathabx}
\usepackage{tikz-cd}
 \usetikzlibrary{matrix}
\usepackage[colorlinks=true, linkcolor=blue, citecolor=magenta]{hyperref}
\usepackage{mathrsfs}
\usepackage{soul}
\usepackage{arydshln}

\newcommand{\cal}{\mathcal}

\newcommand{\chop}{\dagger}

\def\epsilon{\varepsilon}
\def\phi{\varphi}

\def\hat{\widehat}

\def\subset{\subseteq}

\newcommand{\supp}{\mbox{Supp}}

\newcommand{\card}{\mbox{card}}

\newcommand{\bvec}{\overleftarrow}

\newcommand{\R}{\mathbb R}
\newcommand{\Z}{\mathbb Z}

\newcommand{\N}{\mathbb N}
\newcommand{\Hy}{\mathbb H}

\def\strutdepth{\dp\strutbox}
\def \ss{\strut\vadjust{\kern-\strutdepth \sss}}
\def \sss{\vtop to \strutdepth{
\baselineskip\strutdepth\vss\llap{$\diamondsuit\;\;$}\null}}

\def\strutdepth{\dp\strutbox}
\def \sst{\strut\vadjust{\kern-\strutdepth \ssss}}
\def \ssss{\vtop to \strutdepth{
\baselineskip\strutdepth\vss\llap{$\spadesuit\;\;$}\null}}

\def\strutdepth{\dp\strutbox}
\def \ssh{\strut\vadjust{\kern-\strutdepth \sssh}}
\def \sssh{\vtop to \strutdepth{
\baselineskip\strutdepth\vss\llap{$\heartsuit\;\;$}\null}}

\def\qed{\hfill\rlap{$\sqcup$}$\sqcap$\par}

\def\tilde{\widetilde}

\def\strutdepth{\dp\strutbox}
\def \ss{\strut\vadjust{\kern-\strutdepth \sss}}
\def \sss{\vtop to \strutdepth{
\baselineskip\strutdepth\vss\llap{$\diamondsuit\;\;$}\null}}

\def\strutdepth{\dp\strutbox}
\def \sst{\strut\vadjust{\kern-\strutdepth \ssss}}
\def \ssss{\vtop to \strutdepth{
\baselineskip\strutdepth\vss\llap{$\spadesuit\;\;$}\null}}
\def\qed{\hfill\rlap{$\sqcup$}$\sqcap$\par}

\newtheorem{thm}{Theorem}[section]
\newtheorem{cor}[thm]{Corollary}
\newtheorem{lem}[thm]{Lemma}
\newtheorem{prop}[thm]{Proposition}

\newtheorem{fact}[thm]{Fact}
\theoremstyle{definition}
\newtheorem{defn}[thm]{Definition}
\newtheorem{example}[thm]{Example}

\newtheorem{rem}[thm]{Remark}
\newtheorem{defn-rem}[thm]{Definition-Remark}

\newtheorem{entre-nous}[thm]{\large JUSTE ENTRE NOUS}

\newtheorem{assumption}[thm]{Assumption}

\theoremstyle{remark}

\numberwithin{equation}{section}

\begin{document}

\author[N.~Bedaride]{Nicolas B\'edaride}
\author[A.~Hilion]{Arnaud Hilion}
\author[M.~Lustig]{Martin Lustig}

\address{\tt 
Aix Marseille Universit\'e, CNRS, 
I2M UMR 7373,
13453  Marseille, 
France}

\email{\tt Nicolas.Bedaride@univ-amu.fr}

\address{\tt
Institut de Math\'ematiques de Toulouse, UMR 5219, Universit\'e de Toulouse, UPS F-31062
Toulouse Cedex~9, France}

\email{\tt arnaud.hilion@math.univ-toulouse.fr}

\address{\tt 
Aix Marseille Universit\'e, CNRS, 
I2M UMR 7373,
13453  Marseille, 
France}

\email{\tt Martin.Lustig@univ-amu.fr}

\title[Measures transfer and $S$-adic development]{
Measure transfer and $S$-adic developments for subshifts}

\begin{abstract} 
Based on previous work of the authors, to any $S$-adic development of a subshift  $X$ a ``directive sequence'' of commutative diagrams is associated, which consists at every level $n \geq 0$ of the measure cone and the letter frequency cone of the level subshift  $X_n$  associated canonically to the given $S$-adic development.
 
The issuing rich picture enables one to deduce results about $X$  with unexpected directness. For instance, we exhibit a large class of minimal subshifts with entropy zero that all have infinitely many ergodic probability measures.
 
As a side result we also exhibit, for any integer $d \geq 2$, an $S$-adic development of a minimal, aperiodic, uniquely ergodic subshift $X$, where 
all level alphabets $\cal A_n$ have cardinality $d\,$, while none of the $d-2$ bottom level morphisms is recognizable in its level subshift $X_n \subset \cal A_n^\Z$.
\end{abstract} 

\thanks{The first author was partially supported by 
ANR Project IZES ANR-22-CE40-0011}
 
 
 
\maketitle

\section{Introduction}
\label{sec:intro}

A {\em subshift} over a finite {alphabet} $\cal A$ is a non-empty, closed and shift-invariant subset $X \subset \cal A^\Z$.
A very efficient tool to investigate such a subshift $X$ is given by an {\em $S$-adic development} of $X$: the latter is 
obtained 
by a {\em directive sequence} $\bvec \sigma$ of 
monoid morphisms $\sigma_n: \cal A^*_{n+1} \to \cal A_n^*$ for all integers $n \geq 0$, where each $\cal A_n$ is again a finite alphabet, and $\cal A_n^*$ denotes the free monoid over $\cal A_n$. 
The morphisms $\sigma_n$ here are all assumed to be {\em non-erasing}, i.e. none of the letters of $\cal A_{n+1}$ is mapped to the empty word.
The directive sequence $\bvec \sigma$ {\em generates} the given subshift $X$ if for some identification $\cal A = \cal A_0$ any finite factor $x_k \ldots x_\ell$ of any biinfinite word ${\bf x} = \ldots x_{-1} x_0 x_1 \ldots \in X$ is also a factor of some $\sigma_0 \circ \ldots \circ \sigma_{n-1}(a_i)$ with $a_i \in \cal A_{n}$, and conversely: any such ${\bf x}$ belongs to $X$. One usually also assumes that $\bvec \sigma$ is {\em everywhere growing}, which means that 
$\underset{n\to \infty}{\liminf} (\min\{|\sigma_0 \circ \ldots \circ \sigma_{n-1}(a_i)| \mid a_i \in \cal A_{n}\}) = \infty$.
It is well known that any subshift $X \subset \cal A^\Z$ is generated by some everywhere growing directive sequence $\bvec \sigma$.

A directive sequence $\bvec \sigma$ as above determines at every level $n \geq 0$ a {\em level subshift} $X_n \subset \cal A_n^\Z$, which is the subshift generated by the truncated sequence $\bvec \sigma\chop_n$, obtained from $\bvec \sigma$ through forgetting all levels $k < n$ and the corresponding level morphisms. It is a straight forward observation that every level morphism $\sigma_n$ induces a map $X_{n+1} \to X_n$ which is surjective on shift-orbits. 

More generally, any non-erasing morphism $\sigma: \cal A^* \to \cal B^*$ between free monoids over finite alphabets $\cal A$ and $\cal B$ respectively, defines for any subshift $X \subset \cal A^\Z$ an image subshift $\sigma(X)$, and it is natural to ask which properties of $X$ are inherited (under suitable hypotheses) by the image subshift $\sigma(X)$.
In our cousin paper \cite{PartI} we have formally introduced and studied, for any 
such morphism $\sigma$, a {\em measure transfer map} 
$\sigma_X^\cal M: \cal M(X) \to \cal M(\sigma(X))$,  
where $\cal M(X)$ denotes the {\em measure cone} on $X$, i.e. the set of all shift-invariant Borel measures on the subshift $X$. The map $\sigma_X^\cal M$ is the restriction/co-restriction of a map $\sigma^\cal M: \cal M(\cal A^\Z) \to \cal M(\cal B^\Z)$ which 
is linear, functorial and commutes with the support map on subshifts (see section \ref{sec:2.2}).

We thus obtain canonically, for any everywhere growing directive sequence $\bvec \sigma = (\sigma_n: \cal A^*_{n+1} \to \cal A^*_n)_{n \geq 0}$
as above, an induced sequence $ \cal M(\bvec \sigma) = (\sigma^\cal M_n: \cal M(X_{n+1}) \to \cal M(X_n))_{n \geq 0}$ of linear maps 
$\sigma^\cal M_n := \sigma^\cal M_{X_{n+1}}$ on the measure cones $\cal M(X_{n+1})$.

Furthermore, any invariant measure $\mu$ on a subshift $X \subset \cal A^\Z$ defines canonically a {\em letter frequency vector} $\vec v(\mu)$ in the non-negative cone $\R_{\geq 0}^\cal A$ of the vector space $\R^\cal A$, where for each letter $a_i \in \cal A$ the coordinate of $\vec v(\mu)$ is given by the measure $\mu([a_i])$ of the cylinder $[a_i]$. The latter consists of all biinfinite words ${\bf x} \in \cal A^\Z$ as above for which the letter with index 1 satisfies $x_1 = a_i$. The cone of all such letter frequency vectors is denoted by $\cal C(X) \subset \R_{\geq 0}^\cal A\,$; it gives rise to a canonical linear {\em evaluation map} $\zeta_X: \cal M(X) \to \cal C(X)$ which by definition is surjective. 

It has been shown in \cite{PartI} that the linear map $\R^\cal A \to \R^\cal B$, defined by the incidence matrix $M(\sigma)$ of any non-erasing free monoid morphism $\sigma: \cal A^* \to \cal B^*$, commutes via the evaluation maps $\zeta_{\cal A^\Z}$ and $\zeta_{\cal B^\Z}$ with the measure transfer map $\sigma^\cal M$. We thus obtain, for any directive sequence $\bvec \sigma$ as above, a rather useful commutative diagram:
$$\,\,\,\, \ldots \,\,\overset{\sigma_{n+1}^\cal M}{\longrightarrow}\,\,\, \cal M(X_{n+1}) \overset{\sigma_{n}^\cal M}{\longrightarrow}  \,\,\cal M(X_n)\,\,\overset{\sigma_{n-1}^\cal M}{\longrightarrow} \,\,\,\ldots \,\,\,\,\,\,\overset{\sigma_{2}^\cal M}{\longrightarrow} \,\,\cal M(X_1) \overset{\sigma_{1}^\cal M}{\longrightarrow}  \cal M(X)$$
$$\qquad \qquad \qquad\downarrow \zeta_{X_{n+1}} \qquad\,\,\,\,\downarrow \zeta_{X_{n}} \qquad \qquad \qquad \qquad\,\,\,\,\downarrow \zeta_{X_{1}} \qquad \qquad\downarrow \zeta_{X}$$
$$ \ldots \,\,\overset{M(\sigma_{n+1})}{\longrightarrow} \cal C(X_{n+1}) \overset{M(\sigma_{n})}{\longrightarrow}  \cal C(X_n)\overset{M(\sigma_{n-1})}{\longrightarrow} \,\,\ldots \,\,\overset{M(\sigma_{2})}{\longrightarrow} \cal C(X_1) \,\,\overset{M(\sigma_{1})}{\longrightarrow} \,\, \cal C(X)$$

\smallskip

A {\em measure tower} $\bvec \mu = (\mu_n)_{n \geq 0}$ on a directive sequence $\bvec \sigma$ as above, defined by postulating $\mu_n \in \cal M(X_n)$ and $\sigma_n^\cal M(\mu_{n+1}) = \mu_n\,$, 
defines a tower of letter frequency vectors $\vec v(\bvec \mu) = (\vec v(\mu_n))_{n \geq 0}$ which satisfy $M(\sigma_{n+1}) \cdot \vec v(\mu_{n+1}) = \vec v(\mu_n)$. 
This 
last equality had been used in \cite{BHL1} as defining equality for what was called there a {\em vector tower} over the directive sequence $\bvec \sigma$. A linear evaluation map $\frak m: \cal V(\bvec \sigma) \to \cal M(X)$, from the set $\cal V(\bvec \sigma)$ of all such vector towers, to the measure cone $\cal M(X)$ of the subshift $X$
generated by $\bvec \sigma$, has been established in \cite{BHL1}, and the map $\frak m$ is shown in \cite{BHL1} to be always surjective, as long as $\bvec \sigma$ is everywhere growing (but no other hypotheses are needed). We obtain:

\begin{prop}
\label{1.2}
For any everywhere growing directive sequence $\bvec \sigma$ there is a canonical linear bijection between the cone $\cal V(\bvec \sigma)$ of vector towers and the cone $\cal M(\bvec \sigma)$ of measure towers on $\bvec \sigma$, given by the letter frequency map
$$\bvec \mu = (\mu_n)_{n \geq 0} \,\,\mapsto\,\, \bvec v = (\vec v_n)_{n \geq 0}\,,$$
with $\vec v_n = \vec v(\mu_n) =(\mu_n([a_k]))_{a_k \in \cal A_n}$ for all levels $n \geq 0$.
\end{prop}

Based on the main result of our previous paper \cite{BHL1} (quoted below as Theorem \ref{thm1}) we derive from this set-up the following consequence (see Proposition \ref{4.5}):

\begin{thm}
\label{1.1}
For any non-erasing monoid morphism $\sigma: \cal A^* \to \cal B^*$ and any subshift $X \subset \cal A^\Z$ the induced measure transfer map $\sigma^\cal M$ maps the measure cone $\cal M(X)$ of $X$ 
surjectively to the measure cone $\cal M(\sigma(X))$ of the image subshift $\sigma(X)$:
$$\sigma^\cal M(\cal M(X)) = \cal M(\sigma(X))$$
\end{thm}

This general surjectivity result 
for the measure transfer map $\sigma^\cal M$ is mirrored 
in the special case where $\sigma$ is recognizable in $X$ (see Definition \ref{6.1}) 
by the the following fact, proved below in Corollary \ref{3m.2}:

\begin{prop}
\label{1.recog}
If a 
non-erasing morphism $\sigma: \cal A^* \to \cal B^*$ is recognizable in a subshift $X \subset \cal A^\Z$, then the measure transfer map $\sigma^\cal M_X: \cal M(X) \to \cal M(\sigma(X))$ is injective.
\end{prop}

We 
apply 
this injectivity result 
to any directive sequence $\bvec \sigma =(\sigma_n)_{n \geq 0}\,$, where each level map $\sigma_n$ is assumed to be recognizable in the corresponding level subshift $X_{n+1}\,$. Such {\em totally recognizable} directive sequences 
(or slight variations of it) have recently received a lot of attention (see for instance \cite{BPR21}, \cite{BSTY19}, \cite{DDMP}, \cite{Espinoza}), and they are shown to play a central role in the $S$-adic 
approach to 
symbolic dynamics. 
We obtain (see Theorem \ref{4.6d}):

\begin{thm}
\label{1.3}
For any totally recognizable everywhere growing directive sequence $\bvec \sigma$, with generated subshift $X = X_{\tiny \bvec \sigma}$, the linear surjective map of cones
$$\frak m: \cal V(\bvec \sigma) \to \cal M(X)$$
is a bijection.
\end{thm}

We combine this result with a construction from our earlier paper \cite{BHL2}, where for any integer $d \geq 2$ a subshift $X$ with $d$ distinct invariant ergodic probability measures has been shown to exist, 
while $X$ is defined by an everywhere growing directive sequence with level alphabets $\cal A_n$ that all have cardinality $\card(\cal A_n) = d$.
This construction is used in section \ref{sec:Cyr-Kra} below to 
define a large ``diagonal'' family 
$\frak X$ 
of directive sequences $\bvec \sigma$ and to give a quick proof 
(see Theorem \ref{9.2x}) 
that they all generate subshifts $X_{\tiny \bvec \sigma}$ which have 
a remarkable 
property, exhibited first by a quite different and more elaborate method for 
very particular 
subshifts in a recent paper by Cyr-Kra (see \cite{CK2}):

\begin{cor}
\label{1.4}
For any directive sequence $\bvec \sigma \in 
\frak X$ the subshift $X_{\tiny \bvec \sigma}$ is minimal, has topological entropy $h_{X_{\tiny \bvec \sigma}} = 0$ and admits infinitely many distinct ergodic probability measures in $\cal M(X_{\tiny \bvec \sigma})$.
\end{cor}

The directive sequences considered in the last corollary are all totally recognizable, and they are ``large'', in that their {\em alphabet rank}, i.e. the 
limit 
inferior of the cardinality of the level alphabets, is infinite. For finite alphabet rank, on the other hand, the condition ``totally recognizable'' can be replaced by a distinctly weaker condition: in this case, the linear map defined by the incidence matrix $M(\sigma_n)$ is for any sufficiently high level $n \geq 0$ a forteriori (from the surjectivity result in Theorem \ref{1.1}) injective on the subspace spanned by the cone $\cal M(X_n)$. 
In the special -- but rather frequent -- case that this injectivity property of the $M(\sigma_n)$ is also true for all low levels, the bijectivity of the map $\frak m$ as in Theorem \ref{1.3} above is a direct consequence of our set-up. 
We thus 
obtain 
(see Corollary \ref{6.2d}):

\begin{cor}
\label{1.5}
Let $X \subset \cal A^\Z$ be a subshift generated by an everywhere growing directive sequence $\bvec \sigma = (\sigma_n)_{n \geq 0}$ of finite alphabet rank. 
Assume that 
for any $n \geq 0$ the incidence matrix $M(\sigma_n)$ is invertible over $\R$. 
Then any invariant measure $\mu$ on the subshift $X$ 
is determined by 
the letter frequency vector associated to $\mu$, 
i.e. by the values $\mu([a_k])$ for all $a_k \in \cal A$.
\end {cor}

This generalizes a result of \cite{B+Co}, obtained under additional hypotheses by very different methods.

\smallskip

A slightly more general situation than considered in Theorem \ref{1.3}, which deserves some particular interest, occurs if the given directive sequence is only {\em eventually recognizable}, i.e. only for sufficiently high levels one assumes that the level morphisms are recognizable in the corresponding level subshift. In section \ref{sec:pseudo-A} we investigate non-recognizable morphisms, and in particular we show 
in Corollary \ref{2-interesting-sequences} 
the following result, which is somewhat surprising, in view of the claims in \cite{DDMP} and \cite{Espinoza} 
(see 
Remark \ref{last}).

\begin{prop}
\label{1.6}
For any integer $n_0 \geq 0$ there exists an everywhere growing directive sequence $\bvec \sigma = (\sigma_n)_{n \geq 0}$ with the following properties:
\begin{enumerate}
\item
For any $n \geq n_0$ the level alphabets satisfy $\cal A_n = \cal A_{n_0}$ and the level morphisms are stationary: $\sigma_n = \sigma_{n_0}$. 
Furthermore, 
each level morphisms $\sigma_n$ is recognizable in the level subshift $X_{n+1}$.
\item
For any level $n$ with $0 \leq n \leq n_0 -1$ we have 
$\card(\cal A_n) = n+2 = \card(\cal A_{n+1})-1$, 
and none of the level morphisms $\sigma_n$ is recognizable in the level subshift $X_{n+1}$. 
\item
All level subshifts $X_n$ are minimal, uniquely ergodic, and aperiodic. 
\end{enumerate}
(In fact, each level subshift $X_n$ is actually an interval exchange subshift, obtained 
from the 
stable lamination of a pseudo-Anosov homeomorphism on a suitably punctured surface.)
\end{prop}

\medskip
\noindent
{\em Acknowledgements:}
The authors would like to thank Fabien Durand and Samuel Petite 
for encouraging remarks and interesting comments.

%%%%%%%%%%%%

\section{Terminology, notation, conventions and some quotes}
\label{sec:2}

In this section we first recall some standard terminology from symbolic dynamics (see subsection \ref{sec:2.1}), then summarize the notation introduced in \cite{PartI} and some of its results (see subsection \ref{sec:2.2}), and in subsection \ref{sec:2.3} we recall some classical $S$-adic terminology and quote the main result from \cite{BHL1}, which plays a key role later in this paper.

\subsection{Standard terminology from symbolic dynamics}
\label{sec:2.1}

${}^{}$
\smallskip

Throughout this paper we denote by $\cal A, \cal B$ or $\cal C$ non-empty finite sets, called {\em alphabets}, 
and by $\cal A^*, \cal B^*$ or $\cal C^*$ the free monoid over those alphabets. Every element $w \in \cal A^*$
is a word in the {\em letters} $a_1, a_2, \ldots, a_d$ of $\cal A$, i.e. 
$$w = x_1 x_2 \ldots x_n \qquad \text{with} \qquad x_i \in \{a_1, a_2, \ldots, a_d\} = \cal A$$
for any $i = 1, \ldots, n\,$, and the empty word is denoted by $\epsilon$. 
Here $n$ is the {\em length} of $w$, denoted by $|w|$, and one sets $|\epsilon| = 0$. We immediately verify the formula $|w| = \underset{a_j \in \cal A}{\sum} |w|_{a_j}\,$, where $|w|_{a_j}$ denotes the number of occurrences of the letter $a_j$ in $w$. More generally, for any second word $u \in \cal A^*$ we denote by $|w|_u$ the number of (possibly overlapping) occurrences of $u$ as subword $x_k \ldots x_\ell$ (also called a {\em factor}) of $w$.

\smallskip

Any monoid morphism $\sigma: \cal A^* \to \cal B^*$ is determined by the family of letter images $\sigma(a_i) \in \cal B^*$ for all $a_i \in \cal A$,
and this family can be chosen freely.
Such a morphism $\sigma$ is {\em non-erasing} if $|\sigma(a_i)| \geq 1$ for 
all 
$a_i \in \cal A$. All morphisms considered in this paper will be non-erasing. Note that any composition of non-erasing morphisms is non-erasing.

Every monoid morphism $\sigma: \cal A^* \to \cal B^*$ 
induces canonically a linear map $\R_{\geq 0}^\cal A \to \R_{\geq 0}^\cal B\,$, given by the 
{\em incidence matrix}
\begin{equation}
\label{incidence-m}
M(\sigma) = (|\sigma(a_j)|_{b_i})_{b_i \in \cal B, \, a_j \in \cal A}\, .
\end{equation}

\smallskip

To any alphabet $\cal A$ there is also associated the {\em full shift} $\cal A^\Z$; its elements are written as biinfinite words
\begin{equation}
\label{eq2.1.5e}
{\bf x} = \ldots x_{i-1} x_i x_{i+1} \ldots
\end{equation}
with $x_i \in \cal A$ for any index $i \in \Z$. The set $\cal A^\Z$ is naturally equipped with the product topology (with respect to the discrete topology on $\cal A$), 
and $\cal A^\Z$ is a Cantor set unless $\card(\cal A) = 1$.
Furthermore, the space $\cal A^\Z$ comes naturally with a shift-operator 
$T$, defined 
for any ${\bf x}$ as in (\ref{eq2.1.5e}) 
by $T({\bf x}) 
= \ldots y_{i-1} y_i y_{i+1} \ldots$ with $y_i = x_{i+1}$ for any $i \in \Z\,$.
The shift-operator acts as homeomorphism on the space $\cal A^\Z\,$; 
for convenience it will always be denoted by the symbol $T$, independently of 
the choice of the given alphabet $\cal A$. 

For any integers $k \leq l$ we denote by ${\bf x}_{[k, \ell]}$ the subword (again also called {\em factor}) $x_k \ldots x_\ell$ of the biinfinite word ${\bf x}$ as in (\ref{eq2.1.5e}). 
We also consider 
the one-sided infinite {\em positive half-word} ${\bf x}_{[1, \infty)} = x_1 x_2 \ldots$ of $\bf x$.

To any word $w \in \cal A^*$ there is associated the {\em cylinder} $[w] \subset \cal A^\Z\,$, which consists of all words ${\bf x} \in \cal A^\Z$ which satisfy ${\bf x}_{[1, |w|]} = w$. If $w$ is the empty word, then $[w] = \cal A^\Z$. The set of all cylinders $[w]$ together with their shift translates $T^m([w])$ for any $m \in \Z$ constitute a basis for the 
above specified 
topology of the space $\cal A^\Z$.

A non-empty subset $X \subset \cal A^\Z$ is a {\em subshift} if $X$ is closed and if $T(X) = X$. A subshift $X$ is {\em minimal} if none of its subsets is a subshift except $X$ itself. This is equivalent to the statement that for any ${\bf x} \in X$ the {\em shift-orbit} $\cal O({\bf x}) = \{T^m({\bf x}) \mid m \in \Z\}$ is dense in $X$. A minimal subshift $X$ is either uncountably infinite or else it
is finite: in this case $X$ 
consists of the single shift-orbit $X = \cal O(w^{\pm \infty})$ of some periodic word $w^{\pm \infty} = \ldots w w w \ldots\, $, which is well defined for any non-empty $w \in \cal A^*$ by the convention $w^{\pm \infty}_{[1, \infty)} = w w w \ldots\, \, $. It follows that any infinite minimal subshift is 
in particular 
{\em aperiodic}, which means that $X$ doesn't contain any 
periodic word $w^{\pm \infty}$.

Any subshift $X \subset \cal A^\Z$ defines a {\em language} $\cal L(X)$ which consists of all words $w \in \cal A^*$ that occur as factor in some ${\bf x} \in X$. Conversely, 
every infinite subset $\cal L \subset \cal A^*$ {\em generates} a subshift $X(\cal L) \subset \cal A^\Z$, defined by the property that any word from $\cal L(X)$ must occur as factor in some $w' \in \cal L$.

For any subshift $X \subset \cal A^\Z$ and any $n \in \N$ one denotes by $p_X(n)$ the number of words in $\cal L(X)$ of length $n$. The following limit is well defined and is known as {\em topological entropy} $h_X$ of the subshift $X$:
\begin{equation}
\label{eq2.ent}
h_X = \lim_{n \to \infty} \frac{\log p(n)}{n}
\end{equation}

\smallskip
Any non-erasing monoid morphism $\sigma: \cal A^* \to \cal B^*$ defines canonically a map
\begin{equation}
\label{eq2.3e}
\sigma^\Z: \cal A^\Z \to \cal B^\Z
\end{equation}
where for any ${\bf x} \in \cal A^\Z$ the image ${\bf y} = \sigma^\Z({\bf x}) \in \cal B^\Z$ is defined by extending $\sigma$ first to the positive half-word ${\bf x}_{[1, \infty)}$ to define ${\bf y}_{[1, \infty)}$, and subsequently extending $\sigma$ to all of $\bf x$. 

For almost all subshifts $X \subset \cal A^Z$ the image set $\sigma^\Z(X)$ will not be shift-invariant and hence not be a subshift. However, there is a canonical {\em image subshift} $\sigma(X)$ of $X$, which admits several naturally equivalent definitions:

\begin{rem}
\label{image-shift}
The following three definitions of the {image subshift} $Y := \sigma(X)$ are equivalent, for any non-erasing monoid morphism $\sigma: \cal A^* \to \cal B^*$:
\begin{enumerate}
\item
$Y$ is the intersection of all subshifts that contain the set $\sigma^\Z(X)$.

\item
$Y$ is the 
the union of all shift-orbits $\cal O(\sigma({\bf x}))$, 
for any ${\bf x} \in X$. 
[Note here (see Lemma 2.4 of \cite{PartI}) that this union is always closed, 
a fact that 
a priori 
can not be taken for granted.]
\item
$Y$ is the subshift generated by the language $\sigma(\cal L(X))$. Thus $Y$ consists of all biinfinite words ${\bf y} \in \cal B^\Z$ with the property that 
every 
factor of $\bf y$ is also a factor of some word in $\sigma(\cal L(X))$.
\end{enumerate}
\end{rem}

We observe directly the following consequence:

\begin{lem}
\label{2.2e}
Let $\sigma: \cal A^* \to \cal B^*$ be a non-erasing monoid morphism, and let $X \subset \cal A^\Z$ be any subshift. If $\cal L \subset \cal A^*$ is a language that generates $X$, then $\sigma(\cal L)$ generates $\sigma(X)$.
\qed
\end{lem}

\medskip

An {\em invariant measure} on $\cal A^\Z$ is a finite Borel measure $\mu$ on $\cal A^\Z$ which is invariant under the homeomorphism $T$ (= the shift operator). The set of all such invariant measures is denoted by $\cal M(\cal A^\Z)$. For any subshift $X \subset \cal A^\Z$ we denote by $\cal M(X) \subset \cal M(\cal A^\Z)$ the set of those invariant measures $\mu$ for which their support satisfies $\supp(\mu) \subset X$. For notational convenience we identify any such $\mu$ with its restriction to $X$.

Any invariant measure $\mu \in \cal M(\cal A^\Z)$ defines a function
$$\cal A^* \, \to \, \R_{\geq 0}\,\, , \,\,\,\, w \mapsto \mu([w])$$
which for convenience is also denoted by $\mu$, yielding $\mu(w) = \mu([w])$ for any $w \in \cal A^*$. This function is a {\em weight function} in that it satisfies the 
{\em Kirchhoff equalities}
\begin{equation}
\label{Kirchhoff}
\mu(w) = \sum_{a_i \in \cal A} \mu(a_i w) = \sum_{a_i \in \cal A} \mu(w a_i)
\end{equation}
for any $w \in \cal A^*$. 
Conversely, it is well known that any weight function $\mu: \cal A^* \to \R_{\geq 0}$ defines an invariant measure $\mu \in \cal M(\cal A^\Z)$ which satisfies $\mu([w]) = \mu(w)$. The set $\cal M(\cal A^\Z)$ can hence be understood as subset of the infinite dimensional non-negative cone $\R_{\geq 0}^{\cal A^*} = \{\underset{w \in \cal A^*}{\sum} x_w \vec e_{w} \mid x_w \geq 0\}$, from which it inherits the product topology; the latter coincides with the
more generally known {\em weak$^*$-topology} on the measure cone $\cal M(\cal A^\Z)$.

A measure $\mu \in \cal M(\cal A^\Z)$ is a {\em probability measure} if its total mass satisfies $\mu(\cal A^\Z) = 1$. A measure $\mu \in \cal M(\cal A^\Z)$ is  {\em ergodic} if $\mu$ can not be written as 
linear combination with positive coefficients 
of two distinct probability measures. 
For any subshift $X \subset \cal A^\Z$ the number $e(X)$ of ergodic probability measures in $\cal M(X)$ can be finite or infinite; it is equal to the dimension of the linear convex cone $\cal M(X) \subset \R_{\geq 0}^{\cal A^*}$. For any subshift $X \subset \cal A^\Z$ we have $e(X) \geq 1$; if $e(X) = 1$ the subshift $X$ is called {\em uniquely ergodic}.

The support $\supp(\mu)$ of any $\mu \in \cal A^\Z$ is always a subshift $X \subset \cal A^\Z$; if $\mu$ is ergodic, then $X = \supp(\mu)$ is a minimal subshift. The converse conclusion doesn't hold (see 
section \ref{sec:Cyr-Kra} 
below). 

\smallskip

Any 
non-trivial word 
$w \in \cal A^*$ defines a {\em characteristic measure} $\mu_w \in \cal M(\cal A^\Z)$: if $w$ is not
a {proper power}, then $\mu_w$ is given by 
\begin{equation}
\label{eq:charac-m}
\mu_w(B) \,\, := \,\, \card(B \cap \cal O(w^{\pm\infty}))
\end{equation}
for any measurable set $B \subset \cal A^\Z$.
If on the other hand $w = {u}^m$ for some $u \in \cal A^*$ and some integer $m \geq 2$, where $u$ is assumed not to be a proper power, then one has 
$$\mu_w \, \, :=\,\,  m \cdot \mu_{u}$$
In either case, it follows that $\frac{1}{|w|} \mu_w$ is a probability measure.
The set of weighted characteristic measures $\lambda \, \mu_w$ 
(for any $\lambda > 0$) is known to be dense in $\cal M(\cal A^\Z)$.
The support of any characteristic measure is given by 
\begin{equation}
\label{last-minute}
\supp(\mu_w) = \cal O(w^{\pm \infty})\, .
\end{equation}

\smallskip

To any alphabet $\cal A$ one associates canonically the non-negative 
{\em alphabet cone} $\R_{\geq 0}^\cal A = \{\underset{a_k \in \cal A}{\sum} x_k \vec e_{a_k} \mid x_k \geq 0\}$. 
For any invariant measure $\mu$ on $\cal A^\Z$ 
the evaluation on the letter cylinders $[a_k]$ for all $a_k \in \cal A$ defines 
a {\em letter frequency vector}
\begin{equation}
\label{eq2.3.1}
\vec v(\mu) :=\sum_{a_k \in \cal A} \mu([a_k]) \, \vec e_{a_k} \, ,
\end{equation}
so that one has a canonical 
$\R_{\geq 0}$-linear map of cones, 
denoted by
\begin{equation}
\label{eq2.5e}
\zeta_\cal A: \cal M(\cal A^\Z) \to \R_{\geq 0}^\cal A \, , \,\, \mu \mapsto 
\vec v(\mu) \, .
\end{equation}
For any subshift $X \subset \cal A^\Z$ 
the restriction of 
this map to $\cal M(X)$ 
will be denoted by $\zeta_X$.
The image of this map is a cone, denoted by 
\begin{equation}
\label{eq2.3.3}
\cal C(X) := \zeta_X(\cal M(X)) \subset \R_{\geq 0}^\cal A\, ,
\end{equation}
and called the {\em letter frequency cone} of the subshift $X$. For simplicity we will below, for any subshift $X \subset \cal A^\Z$ and any morphism $\sigma: \cal A^* \to \cal B^*\,$, use the symbol $M(\sigma)$ to denote all three linear maps 
\begin{equation}
\label{eq2.8.5e}
\text{$\R^\cal A \to \R^\cal B\, , \,\, \R_{\geq 0}^\cal A \to \R_{\geq 0}^\cal B\,\,$ and $\,\,\cal C(X) \to \cal C(\sigma(X))$}
\end{equation}
defined by the incidence matrix of the 
morphism $\sigma$.

\medskip

More details about these basic facts and some references can be found in section 2 of \cite{PartI}.

%%%%
\subsection{Measures on subshifts via vector towers on directive sequences}
\label{sec:2.3}

${}^{}$
\smallskip

In order to state Theorem \ref{thm1} below, which is the main purpose of this subsection, we first recall some standard notation that is also used later.

\smallskip

A {\em directive sequence} $\bvec \sigma = (\sigma_n)_{n \geq 0}$ consists of {\em level morphisms} 
\begin{equation}
\label{eq2.1d}
\sigma_n: \cal A_{n+1}^* \to \cal A_n^*
\end{equation}
for any {\em level} $n \geq 0$, 
where each $\cal A_n$ is a finite non-empty set, called the {\em level $n$ alphabet}.
We sometimes use the less formal but more suggestive notation
$$\bvec \sigma = \sigma_0 \circ \sigma_{1} \circ \sigma_{2} \circ \ldots$$
to denote a directive sequence.

For any integers $m >  n \geq 0$ we define the {\em telescoped level morphism} 
$$\sigma_{[n, m)} := \sigma_n \circ \sigma_{n+1} \circ \ldots \circ \sigma_{m-1}$$
as well as the level $n$ {\em truncated} directive sequence 
\begin{equation}
\label{eq2.trunc}
\bvec \sigma \chop_{\! n} = (\sigma_k)_{k \geq n} \, .
\end{equation}

\smallskip

Any directive sequence $\bvec \sigma$ as in (\ref{eq2.1d}) above {\em generates} a subshift $X = X_{\tiny \bvec \sigma}$ over the {\em base alphabet} $\cal A_0\,$, defined by the convention that 
${\bf x} \in \cal A_0^\Z$
belongs to $X$ if and only if 
for any finite factor $w$ of $\bf x$ there exists some level $n \geq 1$ and some letter $a_j \in \cal A_n$ such that $w$ is also a factor of $\sigma_{[0, n-1)}(a_j)$.

\smallskip

For any level $n \geq 0$ a directive sequence $\bvec \sigma$ as above defines an {\em intermediate level subshift} 
$X_n \subset \cal A_n^\Z$ 
which is generated by the truncated sequence $\bvec \sigma\chop_{\! n}\,$:
\begin{equation}
\label{eq2.11e}
X_n := X_{\tiny \bvec \sigma\chop_{\! n}}
\end{equation}
The subshift $X_n$ is the {image subshift} of the analogously defined level $n+1$ intermediate subshift $X_{n+1}$ under the morphism $\sigma_n \,$, i.e.:
\begin{equation}
\label{eq2.11.5e}
X_n = \sigma_n(X_{n+1}) \qquad \text{for any level} \qquad n \geq 0
\end{equation}

\smallskip

In this paper we will almost exclusively consider directive sequences which are {\em everywhere growing}, by which we mean that the sequence of 
{\em minimal level letter image lengths}
\begin{equation}
\label{eq2.9e}
\beta_-(n) := \min{\big \{}\,|\sigma_{[0, n-1)}(a_j)| \,\,{\big |} \,\,a_j \in \cal A_n {\big \}}
\end{equation}
tends to $\infty$ for $n \to \infty$. We have (see for instance Proposition 5.10 of \cite{BHL1}):

\begin{fact}
\label{2.1d}
For every subshift $X \subset \cal A^\Z$ there exists an everywhere growing directive sequence $\bvec \sigma$ that generates 
$X$. More precisely, using the notation from (\ref{eq2.1d}), one has
\begin{equation}
\label{more-pr}
\cal A_0 = \cal A \qquad \text{and} \qquad X_{\tiny \bvec \sigma} = X.
\end{equation}
\end{fact}

\begin{rem}
\label{2.11x}
If a directive sequence $\bvec \sigma = (\sigma_n)_{n \geq 0}$ is everywhere growing, it could well be that some of the telescoped level maps $\sigma_{[n,m)}: \cal A_{m+1}^* \to \cal A_n^*$ maps a generator $a_i \in \cal A_{m+1}$ to the empty word in $\cal A_n^*$. 
In this case, since the equivalences from Remark \ref{image-shift} could fail to hold (see Example 5.1 of \cite{3D+St}), it is advisable to admit as ``generated subshift'' $X_{\tiny \bvec \sigma}$ only those ${\bf x} \in \cal A_0^\Z$ which can be lifted to any level of $\bvec \sigma$.

With this alteration 
a simple direct argument shows that 
quotienting out any such ``eventually erased'' letter yields new alphabets $\cal A'_n \subset \cal A_n$ and new level maps $\sigma'_n: {\cal A'}_{n+1}^* \to {\cal A'}_n^*$ such that the issuing directive sequence $\bvec \sigma' = (\sigma'_n)_{n \geq 0}$ is everywhere growing and generates the same subshift as the original sequence $\bvec \sigma$. 
In addition, any of the morphisms $\sigma'_n$ or $\sigma'_{[n, m)}$ is non-erasing.

We can hence (and will from now on quietly) 
assume that any everywhere growing directive sequence consists of non-erasing level maps only.

\end{rem}

\begin{rem}
\label{2.2d}
(1)
A directive sequence $\bvec \sigma$ that generates a subshift $X$ is also called an {\em $S$-adic 
development} (or {\em $S$-adic expansion}) 
of $X$, where $S$ stands sometimes for an (often assumed to be finite) set of 
substitutions 
which contains all level morphisms. 
This concept and in particular the terminology ``$S$-adic'' has been introduced by S. Ferenczi in \cite{F}.
In this context one 
often 
assumes that 
the sequence $\bvec \sigma$ has {\em finite alphabet rank}. By this we mean that 
there is a uniform upper bound to 
the cardinality of any level alphabet, 
so that we can identify all level alphabets with a single finite alphabet $\cal A$.

\smallskip
\noindent
(2) If the set $S$ 
consists 
of a single endomorphism $\sigma: \cal A^* \to \cal A^*$, then the $S$-adic subshift $X$, which is generated by the stationary directive sequence $\bvec \sigma = (\sigma_n)_{n \geq 0}$ with $\sigma_n = \sigma$ for all $n \geq 0$, is called {\em substitutive}. It is important to note that we require here the substitution $\sigma$ (or rather: the above stationary directive sequence $\bvec \sigma$) to be everywhere growing. 
The term ``substitution'' itself is often used synonymous to ``endomorphism of a free monoid'', but sometimes 
(varying) additional conditions are imposed (see for instance
\cite{DP20}).

\smallskip
\noindent
(3)
A very convenient criterion to ensure the condition ``everywhere growing'' is given as follows: 

\begin{enumerate}
\item[($\frak Y$)]
Let $\bvec \sigma = (\sigma_n)_{n \geq 0}$ be a directive sequence, and assume that for every level $n \geq 0$ there exists a level $m > n$ such that the telescoped incidence matrix $M(\sigma_{[n, m)})$ is {\em positive} (i.e. it has all coefficients $> 0$). 
\end{enumerate}
One verifies easily that any directive sequence $\bvec \sigma$ which satisfies the criterion ($\frak Y$) is indeed everywhere growing.

\smallskip
\noindent
(4)
The criterion ($\frak Y$) has another important consequence, namely that the subshift $X$ generated by $\bvec \sigma$ is minimal.
For this conclusion we cite  
Theorem 5.3 of \cite{BD}, proved originally in \cite{Dur00}.
\end{rem}

In \cite{BHL1} to any directive sequence $\bvec \sigma$ as in (\ref{eq2.1d}) there has been associated the set $\cal V(\bvec \sigma)$ of 
{\em vector towers $\bvec v = (\vec v_n)_{n \geq 0}$ 
over
$\bvec \sigma$}.
Such a vector tower\,\footnote{ The terminological specification {\em $\bvec \sigma$-compatible vector tower} used in \cite{BHL2} has been dropped here, as all ``vector towers'' occurring in the present paper satisfy the compatibility condition (\ref{eq2.3.4}) for any $n \geq 0$.}
consists of non-negative vectors 
\begin{equation}
\label{eq2.17.5}
\vec v_n = \underset{a_j \in \cal A}{\sum} 
\vec v_n(a_j) \, \vec e_{a_j} \in \R_{\geq 0}^{\cal A_n}
\end{equation}
that are subject to the {\em compatibility condition}
\begin{equation}
\label{eq2.3.4}
\vec v_n = M(\sigma_n) \cdot \vec v_{n+1}
\end{equation}
for all $n \geq 0$.
It has been shown (see \cite{BHL1}, Remark 9.5) that
for any word $w \in \cal A_0^*$ 
and any such vector tower $\bvec v$
the sequence of sums 
$$\sum_{a_j \,\in \cal A_n} 
\vec v_n(a_j) \, |\sigma_{[0,n)}(a_j)|_w \, 
$$
is bounded above and increasing, as long as $\bvec \sigma$ is everywhere growing (but no other condition is needed). 
This gives:

\begin{rem}
\label{2.6e}
(1)
The value
\begin{equation}
\label{eq2.2}
\mu^{\tiny \bvec v}(w) := \lim_{n \to \infty} \sum_{a_j \,\in \cal A_n}  
\vec v_n(a_j) \, |\sigma_{[0,n)}(a_j)|_w 
\end{equation}
is well defined for any $w \in \cal A_0^*$. It is shown 
in \cite{BHL1}, 
Propositions 7.4 and 9.4, 
that the issuing function $\mu^{\tiny \bvec v}: \cal A_0^* \to \R_{\geq 0}$ satisfies the Kirchhoff equalities (\ref{Kirchhoff}), so that we can summarize:

\smallskip
\noindent
(2)
Any vector tower $\bvec v$ on an everywhere growing directive sequence $\bvec \sigma$ defines via equality (\ref{eq2.2}) an invariant measure on the subshift $X$ generated by $\bvec \sigma$, denoted by $\mu^{\tiny \bvec v} \in \cal M(X)$.
\end{rem}

In terms of $S$-adic language, the main result of \cite{BHL1} translates directly into the following (see also section 3 of \cite{BHL2}):

\begin{thm}[{\cite{BHL1}}]
\label{thm1}
Let $\bvec \sigma = (\sigma_n)_{n \geq 0}$ be an everywhere growing directive sequence 
which generates the  
subshift $X :=X_{\tiny \bvec \sigma}$. 
Then the map
$$
{\frak m}_{\tiny \bvec \sigma}
: \cal V(\bvec \sigma) \to \cal M(X)\,,\,\, \bvec v \mapsto \mu^{\tiny\bvec v}$$
is linear and surjective.
\qed
\end{thm}

For any of the level alphabets $\cal A_n$ of a directive sequence $\bvec \sigma$ as above we consider the projection map\,\footnote{\, The map $pr_n$ was denoted in \cite{BHL1} and \cite{BHL2} by $\frak m_n$, but we decided to reserve this notation here for the more telling maps introduced below in section \ref{sec:4}.}
of the set of vector towers to the corresponding non-negative alphabet cone:
$$pr_n: \cal V(\bvec \sigma) \to \R_{\geq 0}^{\cal A_n}\, , \,\, \bvec v = (\vec v_n)_{n \geq 0} \mapsto \vec v_n$$
On the base level $n=0$ this projection splits over the 
evaluation map $\zeta_{\cal A_0}$ 
from (\ref{eq2.5e}) 
via the map $\frak m_{\tiny \bvec \sigma}$ from the last theorem. More precisely, 
this gives (see \cite{BHL1}, Proposition 10.2 (1) and (2)):

\begin{prop}
\label{2.4d}
For any subshift $X \subset \cal A^\Z$, generated by an everywhere growing directive sequence $\bvec \sigma$ as 
in (\ref{eq2.1d}) and (\ref{more-pr}), 
one has:
\begin{enumerate}
\item
The map $\zeta_X: \cal M(X) \to \R_{\geq 0}^{\cal A}\,, \,\, \mu \mapsto (\mu([a_k])_{a_k \in \cal A}$ satisfies:
$$pr_0 = \zeta_X \circ {\frak m}_{\tiny \bvec \sigma}$$
\item
In particular, for the letter frequency cone $\cal C(X) = \rm{im}(\zeta_X)$ 
(see (\ref{eq2.3.3})) 
this gives:
$$\cal C(X) = \zeta_X({\frak m_{\tiny \bvec \sigma}}(\cal V(\bvec\sigma)))$$ 
\item
Alternatively, the letter frequency cone is obtained as nested intersection as follows:
$$\cal C(X) := \bigcap_ {n \geq 1} \, M(\sigma_{[0,n)})(\R_{\geq 0}^{\cal A_{n}})$$
\item
In particular, 
$\dim \cal C(X)$ 
is a lower bound to the number $e(X)$ of distinct ergodic probability measures on $X$.
\qed
\end{enumerate}
\end{prop}

The following statement is the translation of Remark 9.2 (3) of \cite{BHL1} into the terminology used here.

\begin{lem}
\label{2.8e}
For any vector tower $\bvec v = (\vec v_n)_{n \geq 0}$ over an everywhere growing directive sequence $\bvec \sigma$ as in (\ref{eq2.1d})  one has
$$\lim_{n \to \infty} \sum_{a_j \in \cal A_n} \vec v_n(a_j)  = 0\, ,$$
where the coefficient $\vec v_n(a_j) \in \R_{\geq 0}$ is defined in equality (\ref{eq2.17.5}).
\qed
\end{lem}

%%%%%%%%%%%%
\section{The measure transfer and its injectivity for recognizable morphisms}

In this section we will first recall the definition of the measure transfer map and quote some basic properties derived in \cite{PartI} (see subsection \ref{sec:2.2} 
below), then recall the definition and some related properties of recognizable morphisms (see subsection \ref{sec:3m.2}
below), and in subsection \ref{sec:3m.3} we will derive the injectivity result from the title of this section.

\subsection{The measure transfer and some results from \cite{PartI}}
\label{sec:2.2}

\smallskip

For any non-erasing monoid morphism $\sigma: \cal A^* \to \cal B^*$ we define the {\em subdivision alphabet} 
$\cal A_\sigma = \{a_i(k) \mid a_i \in \cal A \text{ and } 1 \leq k \leq |\sigma(a_i)|\}$. The morphism $\sigma$ now defines a {\em subdivision morphism} $\pi_\sigma: \cal A^* \to \cal A_\sigma^*$ and a {\em letter-to-letter morphism} $\alpha_\sigma: \cal A_\sigma^* \to \cal B^*$, given for any $a_i \in \cal A$ and any $a_i(k) \in \cal A_\sigma$ by: 
$$\pi_\sigma(a_i) = a_i(1) \, a_i(2) \ldots a_i(|\sigma(a_i)|) \qquad \text{and} \qquad \alpha_\sigma(a_i(k)) = [\sigma(a_i)]_k$$
Here by $[\sigma(a_i)]_k$ we mean the $k$-th letter of the word $\sigma(a_i) \in \cal B^*$. 
We obtain 
directly:

\begin{fact}
\label{2.21} For any non-erasing monoid morphism $\sigma: \cal A^* \to \cal B^*$ one has:
$$\sigma \,\, = \,\, \alpha_\sigma \circ \pi_\sigma$$
\end{fact}

For any word $w \in \cal A_\sigma^*$ we denote by $\hat w \in \cal A^*$ the shortest word such that $\pi_\sigma(\hat w)$ contains $w$ as factor. If such $\hat w$ exists, it is 
unique; otherwise we treat $\hat w$ as formal symbol and we set 
\begin{equation}
\label{eq2.X}
\mu(\hat w) = 0
\end{equation}
for any $\mu \in \cal M(\cal A^\Z)$.

For any measure $\mu \in \cal M(\cal A^\Z)$ a measure $\mu^{\pi_\sigma} \in \cal M(\cal A_\sigma^\Z)$ is defined in section 3.1 of \cite{PartI} by 
setting $\mu^{\pi_\sigma}([w]) := \mu([\hat w])$, where $[\hat w]$ is the cylinder associated to the word $\hat w$ (see subsection \ref{sec:2.1}).
On the other hand, for any measure $\mu' \in \cal M(\cal A_\sigma^\Z)$ 
the classical push-forward measure $(\alpha_\sigma)_*(\mu')$ is an invariant measure on $\cal B^\Z$, since $\alpha_\sigma$ is letter-to-letter. 
We thus obtain (see 
\cite{PartI}, section 3):

\begin{thm}
\label{2.2.2}
Let $\sigma: \cal A^* \to \cal B^*$ be a non-erasing morphism of free monoids.

\smallskip
\noindent
(1)
For any invariant measure $\mu$ on $\cal A^\Z$ an invariant measure $\mu^\sigma$ on $\cal B^\Z$ is given by
$$\mu^\sigma = (\alpha_\sigma)_*(\mu^{\pi_\sigma})\, .$$

\smallskip
\noindent
(2)
For any word $w' \in \cal B^*$ the ``transferred measure'' $\mu^\sigma$ takes on the cylinder $[w']$ the value
$$\mu^\sigma([w']) = \sum_{w_i \in \alpha_\sigma^{-1}(w')} \mu([\hat w_i])\,.$$

\smallskip
\noindent
(3)
The issuing measure transfer map
$$\sigma^\cal M: \cal M(\cal A^\Z) \to \cal M(\cal B^\Z)\, , \,\, \mu \mapsto \mu^\sigma$$
induced by the morphism $\sigma$ has the following properties:
\begin{enumerate}
\item[(3a)]
The map $\sigma^\cal M$ is linear (over $\R$) and continuous (with respect to the weak$^*$-topology).
\item[(3b)]
The map $\sigma^\cal M$ is functorial.
\item[(3c)]
If $X$ is the support of $\mu$, then $\sigma(X)$ is the support of $\mu^\sigma$. Hence $\sigma^\cal M$ induces in particular 
on any subshift $X \subset \cal A^\Z$ 
a restriction/co-restriction map
$$
\sigma^\cal M_X: \cal M(X) \to \cal M(\sigma(X)) \, .
$$
\qed
\end{enumerate}
\end{thm}

We also list the following more technical properties derived in \cite{PartI}:

\begin{prop}
\label{2.2.3}
Let $\sigma: \cal A^* \to \cal B^*$ be a non-erasing free monoid morphism, and let $\sigma^\cal M$ be the induced transfer map on the measure cones. Let $\mu \in \cal M(\cal A^\Z)$ be an invariant measure on the full shift $\cal A^\Z$, and denote as before by $\mu^\sigma = \sigma^\cal M(\mu)$ the transferred measure on $\cal B^\Z$.
Then one has:
\begin{enumerate}
\item[(a)]
The total mass of the transferred measure $\mu^\sigma$ is given by the formula
$$\mu^\sigma(\cal B^\Z) =  
\sum_{a_k \in \cal A} \sum_{b_j \in \cal B} |\sigma(a_k)|_{b_j}\cdot \mu(a_k)\, .$$
In particular, if $\mu$ is a probability measure, then in general $\mu^\sigma$ will not be probability.
\item[(b)]
For any generator $b_j \in \cal B$ we have:
$$\mu^\sigma([b_j]) =  \sum_{a_k \in \cal A} |\sigma(a_k)|_{b_j} \cdot \mu(a_k)$$
In particular, for the letter frequency vectors 
from (\ref{eq2.3.1}) we obtain:
\begin{equation}
\label{eq2.3.2}
\vec v(\mu^\sigma) \,\, = \,\, M(\sigma) \cdot \vec v(\mu)
\end{equation}
In other words 
(see Proposition 4.5 of \cite{PartI}), 
the measure transfer map $\sigma^\cal M$
commutes via the evaluation maps $\zeta_\cal A$ and $\zeta_\cal B$ 
from
(\ref{eq2.5e})
with the linear map induced by $\sigma$ on the non-negative 
cone $\R_{\geq 0}^\cal A \,$:
$$\zeta_\cal B \circ \sigma^\cal M \,\, = \,\, M(\sigma) \circ \zeta_\cal A$$
\item[(c)]
For any $w \in \cal A^*$ the cylinder measures satisfy:
$$\mu^\sigma([\sigma(w)]) \,\, \geq \,\, \mu([w])$$
\item[(d)]
For any word $w \in \cal A^*$ the characteristic measure $\mu_w$ satisfies:
$$
\sigma^\cal M(\mu_w) = \mu_{\sigma(w)}$$
\qed
\end{enumerate}
\end{prop}

It remains to 
quote a useful 
evaluation technique for the transferred measure, derived in section 4 of \cite{PartI} from what is stated above as part (2) of Theorem \ref{2.2.2}. For this purpose we define for any non-erasing morphism $\sigma: \cal A^* \to \cal B^*$ and any $w \in \cal A^*, \, u \in \cal B^*$ 
the number $\lfloor\sigma(w) \rfloor_{u}$  of {\em essential occurrences} of $u$ 
in $\sigma(w)$, by which we mean that the first letter of $u$ occurs in the $\sigma$-image of first letter of $w$, and the last letter of $u$ occurs in the $\sigma$-image of last letter of $w$.
By $\langle \sigma \rangle$ we denote the smallest length of any of the letter images $\sigma(a_i)$.

\begin{prop}(\cite{PartI}, {\rm Proposition 4.2})
\label{2.8.formula}
Let $\sigma: \cal A^* \to \cal B^*$ be any non-erasing monoid morphism, and let $\mu \in \cal M(\cal A^\Z)$. Then for any $w' \in \cal B^*$ with $|w'| \geq 2$ the transferred measure $\mu^\sigma = \sigma^\cal M(\mu)$ takes on the cylinder $[w']$ the value
$$
\mu^\sigma([w']) = \sum_{{\big \{}w_j \in \cal A^* \,{\big |}\, |w_j| \leq \frac{|w'|-2}{\langle\sigma\rangle}+2{\big \}}}{\lfloor\sigma(w_j) \rfloor}_{w'} \cdot \mu([w_j])\, .
$$
\qed
\end{prop}

%%%%%%%%
\subsection{Recognizable morphisms and some related properties}
\label{sec:3m.2}

${}^{}$
\smallskip

The following notion has become more and more central to symbolic dynamics 
(see for instance \cite{BSTY19}, \cite{DDMP}, \cite{DLR} or \cite{DP20}):

\begin{defn}
\label{6.1}
Let $\sigma:\cal A^* \to \cal B^*$ be a non-erasing morphism, and let $X \subset \cal A^\Z$ be a subshift over $\cal A$. 
Then $\sigma$ is said to be 
{\em recognizable in $X$}
if the following conclusion is true:

Consider biinfinite words ${\bf x}, {\bf x'} \in X \subset \cal A^\Z$ and ${\bf y} \in \cal B^\Z$ which satisfy:
\begin{enumerate}
\item[(*)]
${\bf y} = T^k(\sigma^\Z({\bf x}))$ and ${\bf y} = T^\ell(\sigma^\Z({\bf x'}))$ for some integers $k, \ell$ which satisfy $0 \leq k \leq |\sigma(x_1)|-1$  and $ 0 \leq  \ell \leq |\sigma(x'_1)|-1$, where $x_1$ and $x'_1$ are the first letters of the 
positive half-words 
${\bf x_{[1, \infty)}} = x_1  x_2 \ldots$ of ${\bf x}$ and ${\bf x'_{[1, \infty)}} = x'_1 x'_2 \ldots$ of ${\bf x'}$ respectively.
\end{enumerate}

Then one has ${\bf x} = \bf x'$ and $k = \ell$.
\end{defn}

As we will see in the next subsection, recognizability in a subshift is much related to the following:

\begin{defn}[\cite{PartI}, Section 5]
\label{recognizable-e}
For any non-erasing monoid morphism $\sigma: \cal A^* \to \cal B^*$ and any subshift $X \subset \cal A^\Z$ we define the following two properties:
\begin{enumerate}
\item
$\sigma$ is {\em shift-orbit injective in $X$}: Any $\bf x$ and $\bf y$ in $X$ have images $\sigma({\bf x})$ and $\sigma({\bf y})$ in the same shift-orbit if and only $\bf x$ and $\bf y$ lie in a common shift-orbit.
\item
$\sigma$ is {\em shift-period preserving in $X$}: For any periodic biinfinite word $w^{\pm \infty} = \ldots w w w \ldots \in X$ the word $w$ can be written as proper power if and only if $\sigma(w)$ can be written as proper power.
${}^{}$
\end{enumerate}
Here 
$w \in \cal A^* \smallsetminus \{\epsilon\}$ is a {\em proper power}\,\footnote{
Elements in $\cal A^*$ which are not a proper power are sometimes called ``primitive''. However, since $\cal A^*$ is canonically embedded into the free group $F(\cal A)$, where the notion of ``primitive elements'' is classical, 
but has a different meaning,
we believe it is better not to use this terminology for a different purpose.}
 if $w = u^m$ for some $u \in \cal A^*$ and some integer $m \geq 2$.
\end{defn}

The following useful property is a direct consequence of the previous definition (see Lemma 
5.2 of \cite{PartI}).

\begin{lem}
\label{2.3e}
Let $\sigma_1: \cal A^* \to \cal B^*$ and $\sigma_2: \cal B^* \to \cal C^*$ be two non-erasing morphisms, and consider a subshift $X \subset \cal A^\Z$ as well as its image subshift $Y = \sigma_1(X) \subset \cal B^\Z$. Then we have:
\begin{enumerate}
\item
The composed morphism $\sigma_2 \circ \sigma_1: \cal A^* \to \cal C^*$ is shift-orbit injective in $X$ if and only if $\sigma_1$ is shift-orbit injective in $X$ and $\sigma_2$ is shift-orbit injective in $Y$.
\item
The composed morphism $\sigma_2 \circ \sigma_1: \cal A^* \to \cal C^*$ is shift-period preserving in $X$ if and only if $\sigma_1$ is shift-period preserving in $X$ and $\sigma_2$ is shift-period preserving in $Y$.
\qed
\end{enumerate}
\end{lem}

%%%%%%%%
\subsection{Injectivity of the measure transfer for recognizable morphisms}
\label{sec:3m.3}

${}^{}$
\smallskip

Let $\sigma: \cal A^* \to \cal B^*$ be a non-erasing morphism of free monoids, and let $\pi_\sigma: \cal A^* \to \cal A_\sigma^*$ and $\alpha_\sigma: \cal A_\sigma^* \to \cal B^*$ be the canonical subdivision morphism and the induced letter-to-letter morphism associated to $\sigma$ which satisfy 
$\sigma = \alpha_\sigma \circ \pi_\sigma$ 
(see Fact \ref{2.21}). 
For any subshift $X \subset \cal A^\Z$ we consider the image subshift $\pi_\sigma(X) \subset \cal A_\sigma^\Z$ and the induced restriction/co-restriction
$$\alpha_\sigma^X: \pi_\sigma(X) \to \sigma(X)$$
of the map $\alpha_\sigma^\Z: \cal A_\sigma^\Z \to \cal B^\Z$ to $\pi_\sigma(X)$ and $\sigma(X)$ respectively.

\begin{prop}
\label{3m.1}
For any non-erasing morphism $\sigma: \cal A^* \to \cal B^*$ and any subshift $X \subset \cal A^\Z$ the following statements are equivalent:
\begin{enumerate}
\item
$\sigma$ is recognizable in $X$.
\item
$\alpha_\sigma^X$ is an isomorphism of subshifts.
\item
$\alpha_\sigma$ is shift-orbit injective and shift-period preserving in $\pi_\sigma(X)$.
\item
$\sigma$ is shift-orbit injective and shift-period preserving in $X$.
\end{enumerate}
\end{prop}

\begin{proof}
We first note that by definition $\alpha_\sigma^X$ is continuous and surjective, so that claim (2) is equivalent to stating that $\alpha_\sigma^X$ is injective. 

Next we observe that claim (1) is equivalent to stating that $\alpha_\sigma^X$ is recognizable in $\pi_\sigma(X)$. This is a direct consequence of the product decomposition $\sigma = \alpha_\sigma \circ \pi_\sigma$ from Fact \ref{2.21} and 
of Lemma 3.5  of \cite{BSTY19}, 
since every subdivision morphism $\pi_\sigma$ is recognizable in the full shift, as follows directly from the definition of $\pi_\sigma$.

In order to show the equivalence (1) $\Longleftrightarrow$ (2) 
we apply Definition \ref{6.1} to the morphism $\alpha_\sigma$ and the subshift $\pi_\sigma(X)\, $:
we observe that, since $|\alpha_\sigma(x)| = 1$ for any letter $x \in \cal A_\sigma\,$, in the hypothesis (*) of Definition \ref{6.1} the integers $k$ \and $\ell$ are necessarily equal to 0. But in this case the conclusion ${\bf x} = {\bf x}'$ stated there amounts precisely to assuring that the map $\alpha_\sigma^\Z$ is injective on $\pi_\sigma(X)$, or in other words, that $\alpha_\sigma^X$ is injective.

\smallskip

The equivalence (2) $\Longleftrightarrow$ (3) is immediate, since 
any subshift-isomorphism preserves orbits and shift-periods, while conversely, any shift-orbit injective letter-to-letter morphism could only fail to be injective if on some periodic orbit the shift-period is not preserved.

\smallskip

Finally, the equivalence (3) $\Longleftrightarrow$ (4) is a direct consequence of 
Lemma \ref{2.3e}, since every subdivision morphism $\pi_\sigma$ is shift-orbit injective and shift-period preserving in the full shift (see Lemma 
5.3  of \cite{PartI}).
\end{proof}

Note that the equivalence of the statements (1) and (2) from Proposition \ref{3m.1} has already been observed in 
\cite{DP20}, Proposition 2.4.24. Indeed, Fabien Durand has suggested to us to use this equivalence in order to derive the following corollary. 
In the mean time we have obtained a result which is actually a bit stronger: 
it turns out 
(see Theorem 
5.5 of \cite{PartI})
that the hypothesis ``shift-orbit injective'' suffices to obtain the same conclusion  
as stated in Corollary \ref{3m.2} below,
but the proof is much less direct.

\begin{cor}
\label{3m.2}
For any non-erasing morphism $\sigma: \cal A^* \to \cal B^*$ and any subshift $X \subset \cal A^\Z$ the measure transfer map $\sigma_X^\cal M: \mu \to \mu^\sigma$ is injective if $\sigma$ is recognizable in $X$.
\end{cor}

\begin{proof}
We decompose $\sigma = \alpha_\sigma \circ \pi_\sigma$ as in Fact \ref{2.21}, so that from the functoriality of the measure transfer (see property (3b) of Theorem \ref{2.2.2}) we have $\sigma_X^\cal M = (\alpha_\sigma^X)^\cal M \circ \pi_\sigma^\cal M$. The injectivity of $\pi_\sigma^\cal M$ is immediate from the definition of a subdivision morphism (see Lemma 
5.4 of \cite{PartI}), and the injectivity of $(\alpha_\sigma^X)^\cal M$ is a direct consequence of Proposition \ref{3m.1} (2).
\end{proof}

\begin{rem}
\label{added-late}
Consider any 
non-erasing morphism $\sigma: \cal A^* \to \cal B^*$ and any subshift $X \subset \cal A^\Z$ with image subshift $Y = \sigma(X) \subset \cal B^\Z$.

\smallskip
\noindent
(1)
Assume that the subshift $Y$ contains a periodic word $w^{\pm \infty}$ for some $w \in \cal B^* \smallsetminus \{\epsilon\}$, and that the morphism $\sigma$ is shift--orbit injective. Then, in order for $\sigma$ to be shift-period preserving in $X$, a necessary condition is that at least one of the letters $a_i \in \cal A$ satisfies $|\sigma(a_i)| \leq |w|$. 

As a consequence, unless a given subshift $Y$ is aperiodic, in any everywhere growing $S$-adic development of $Y$ there will always be infinitely many level morphisms which are not recognizable in their corresponding level subshift.

\smallskip
\noindent
(2)
This has sparked the following weakening of the notion of ``recognizability'' which has become recently very popular (see for instance 
\cite{3D+St}).

The morphism $\sigma$ is said to be {\em recognizable for aperiodic points in $X$} if the conclusion in Definition \ref{6.1} holds under the strengthened assumption that ${\bf y}$ is not a periodic word. 

\smallskip
\noindent
(3)
From the above proof of Proposition \ref{3m.1} we observe that the property ``shift-orbit injective in $X$'' implies the property ``recognizable for aperiodic points in $X$''. 

Indeed, since the subdivision morphism $\pi_\sigma$ is always shift-orbit injective and shift-periodic preserving (and thus recognizable) in the full shift, the property ``$\sigma$ is recognizable for aperiodic points in $X$'' is equivalent to ``$\alpha_\sigma^X$ is recognizable for aperiodic points in $\pi_\sigma(X)$''.
This in turn is equivalent to stating that every non-periodic word in $\sigma(X)$ has precisely one preimage under the letter-to-letter map $\alpha_\sigma$. But since we assume that $\sigma$ and hence $\alpha_\sigma^X$ is shift-orbit injective, two distinct such preimages must lie in the same shift-orbit, which implies that their image in $\sigma(X)$ must be periodic.
\end{rem}

%%%%%%%%%%%%%%%%%%
\section{The measure transfer via vector towers}
\label{sec:3}

In this section we will consider a subshift $X$ given by means of a directive sequence, an invariant measures $\mu$ on $X$ given by means of a vector tower on this directive sequence, and a morphism $\tau: X \to Y = \tau(X)$ which we use to build a new directive sequence for $Y$ by simply adding $\tau$ at the bottom to the given sequence. Then the given vector tower is naturally transferred to a new vector tower on the new directive sequence, and, as do all vector towers, it defines an invariant measure $\mu'$ on the subshift $Y$ generated by this new sequence. The main goal of this section is to show that the new measure $\mu'$ is precisely the image of given measure $\mu$ under the transfer map $\tau^\cal M$ induced by the morphism $\tau$ 
(see Theorem \ref{2.2.2} (1)).

For convenience we summarize the running hypotheses for this section as follows:

\begin{assumption}
\label{3.0.2}
Let $\tau: \cal A^* \to \cal B^*$ be a non-erasing morphism of free monoids over finite alphabets $\cal A$ and $\cal B$, and let $\bvec \sigma = 
(\sigma_n)_{n \geq 0}$ be 
an everywhere growing directive sequence with 
base level 
alphabet $\cal A_0 = \cal A$. Let $X := X_{\tiny \bvec \sigma} \subset \cal A^\Z$ be the subshift generated by $\bvec \sigma$, and denote by $Y := \tau(X)$ the image subshift of $X$ given by the morphism $\tau$ (see 
Remark \ref{image-shift}).
\end{assumption}

\begin{defn-rem}
\label{3.0.5d}
Let $\tau$ and $\bvec \sigma = (\sigma_n)_{n \geq 0}$ as well as the subshifts $X$ and $Y$ be as in Assumption \ref{3.0.2}.

\smallskip
\noindent
(1)
We define 
a second 
``prolonged'' 
directive sequence 
$\bvec \sigma^\tau = 
(\sigma'_n)_{n \geq 0}$
by setting $\sigma'_n := \sigma_{n-1}$ for any level $n \geq 1$ and $\sigma'_0 := \tau$, and observe 
from Lemma \ref{2.2e} 
that the subshift 
$X_{\tiny \bvec \sigma^\tau}$ generated by 
$\bvec \sigma^\tau$ agrees precisely with the $\tau$-image subshift $Y = \tau(X) \in \cal B^\Z$.

\smallskip
\noindent
(2)
Consider now any vector tower $\bvec v = (\vec v_n)_{n \geq 0}$ over $\bvec \sigma$, and let $\mu = \frak m_{\tiny\bvec \sigma}(\bvec v)$ be the invariant measure on $X$ associated to $\bvec v$ via Theorem \ref{thm1}.
We obtain 
a prolonged 
vector tower 
$\bvec v^\tau = (\vec v\, '_{\! n})_{n \geq 0}$ over 
$\bvec \sigma^\tau$ by setting $\vec v\, '_{\! n} = \vec v_{n-1}$ for any level $n \geq 1$, and by setting $\vec v\, '_{\! 0} := M(\tau) \cdot \vec v_0\,$. Let $\mu'$ be the associated measure on $Y$, i.e.
\begin{equation}
\label{eq3.0.5d}
\mu' =  
\frak m_{\tiny\bvec \sigma^\tau}(\bvec v^\tau) \, .
\end{equation}
\end{defn-rem}

We can now 
link up 
the measure transfer map defined and studied in \cite{PartI} 
with 
the technology of vector towers from our previous papers \cite{BHL1},\cite{ BHL2}. The following will be the basis for all results presented in this paper:

\begin{prop}
\label{3.1d}
Let $\tau, \bvec \sigma$ and $X$ be as in Assumption \ref{3.0.2}, and
let $\bvec v = (\vec v_n)_{n \geq 0}$ be a vector tower over $\bvec \sigma$, 
with associated invariant measure $\mu = \frak m_{\tiny\bvec \sigma}(\bvec v)$ on $X$. 
Let 
$\bvec \sigma^\tau = 
(\sigma'_n)_{n \geq 0}\,$,  
$\bvec v^\tau = (\vec v\, '_{\! n})_{n \geq 0}$ 
and $\mu' = 
\frak m_{\tiny\bvec \sigma^\tau}(\bvec v^\tau)$
be as in Definition-Remark \ref{3.0.5d}. 

Then the measure transfer map $\tau^\cal M: \cal M(\cal A^\Z) \to \cal M(\cal B^\Z)$ induced by the morphism $\tau$ satisfies:
$$\mu' = \tau^\cal M(\mu) \,\,\, [= \mu^\tau]$$
\end{prop}

\begin{proof}
In this proof we will freely use the terminology from \cite{PartI} as recalled in section \ref{sec:2.2}.

For any word $w' \in \cal B^*$ we consider in the subdivision monoid $\cal A_\tau^*$ the subset 
$W(w')$ of preimages $w_i$ of $w'$ under the induced letter-to-letter morphism $\alpha_\tau: \cal A_\tau^* \to \cal B^*$. 
For each $w_i \in W(w')$ consider 
(as in the paragraph subsequent to Fact \ref{2.21}) 
the word $\hat w_i \in \cal A^*$ 
defined by the conditions that (a) 
its canonically subdivided image $\pi_\tau(\hat w_i)$ contains $w_i\,$, and that (b) the word $\hat w_i$ is shortest among all words in $\cal A^*$ 
which satisfy (a). 
Recall from (\ref{eq2.X}) that either $\hat w_i$ exists and is unique, or else we formally set $\mu(\hat w_i) = 0$ for any $\mu \in \cal M(\cal A^\Z)$. 
From Theorem \ref{2.2.2} (2) 
we know that $\mu^\tau(w') = \underset{w_i \in W(w')}{\sum} \mu(\hat w_i)\,$ 
(with $\mu^\tau = \tau^\cal M(\mu)$ as before).

For the purpose of using formula 
(\ref{eq2.2}) 
we consider now the value of the approximating sum on its right hand side 
of this formula 
for any (large) 
level 
$n-1$, for each of the words $\hat w_i$ and the given vector tower $\bvec v$, i.e. the term 
(see (\ref{eq2.17.5}) for the notation)
\begin{equation}
\label{eq3.1d}
\sum_{a \,\in \cal A_{n-1}} \vec v_{n-1}(a) 
\, 
|\sigma_{[0,n-1)}(a)|_{\hat w_i}\, .
\end{equation}
We sum up the results of (\ref{eq3.1d}) over all $w_i \in W(w')$ 
to get
\begin{equation}
\label{eq3.1.5}
\sum_{w_i \in W(w')}\sum_{a \,\in \cal A_{n-1}} \vec v_{n-1}(a) 
\, 
|\sigma_{[0,n-1)}(a)|_{\hat w_i}\, ,
\end{equation}
and compare the obtained sum to 
the limit on the right hand side of 
(\ref{eq2.2}), 
when this formula is applied 
to $w'$ and to the vector tower 
$\bvec v^\tau$ 
over the prolonged directive sequence $\bvec \sigma^\tau$. 
The $n$-th term of this limit gives 
the sum
\begin{equation}
\label{eq3.2d}
\sum_{a \,\in \cal A_{n-1}} \vec v\,'_{\! n}(a) 
\, 
|\sigma'_{[0,n)}(a)|_{w'}\, .
\end{equation}

When comparing the two sums (\ref{eq3.1.5}) and (\ref{eq3.2d}) 
we keep in mind that according to the set-up from Definition-Remark \ref{3.0.5d} for any $n \geq 1$ we have $\vec v\,'_{\! n}(a) = \vec v_{n-1}(a)$ for any $a \in \cal A_{n-1}\,$, as well as $\sigma'_{[0,n)} = \tau \circ \sigma_{[0,n-1)}\,$.

We now notice that each occurrence of any of the $\hat w_i$ in any of the image words $\sigma_{[0, n-1)}(a)$ with $a \in \cal A_{n-1}$ defines precisely an occurrence of $w_i$ in $\pi_\tau(\sigma_{[0, n-1)}(a))$, and thus an occurrence of $w'$ in $\alpha_\tau(\pi_\tau(\sigma_{[0, n-1)}(a))) = \tau(\sigma_{[0, n-1)}(a)) = \sigma'_{[0, n)}(a)$. 
Furthermore, two distinct occurrences of $\hat w_i$ in some $\sigma_{[0, n-1)}(a)$ define distinct occurrences of $w_i$ in $\pi_\tau(\sigma_{[0, n-1)}(a))$, and thus distinct occurrences of $w'$ in $\sigma'_{[0, n)}(a)$. The same is 
true for occurrences of distinct $\hat w_i$ in $\sigma_{[0, n-1)}(a)$.
It follows (using the above recalled equality $\vec v\, '_{\! n} = \vec v_{n-1}$) that 
\begin{equation}
\label{eq3.5e}
\sum_{w_i \in W(w')} \, \sum_{a \,\in \cal A_{n-1}} \vec v_{n-1}(a) 
\, 
|\sigma_{[0,n-1)}(a)|_{\hat w_i}
\leq
\sum_{a \,\in \cal A_{n-1}} \vec v\, '_n(a) 
\, 
|\sigma'_{[0,n)}(a)|_{w'}
\end{equation}

But the opposite inequality is also true, up to a constant $K_n$ which only depends on $\bvec \sigma$ and not on 
$\bvec v\,$: 
\begin{equation}
\label{eq3.6e}
\sum_{a \,\in \cal A_{n-1}} \vec v\, '_n(a) 
\, 
|\sigma'_{[0,n)}(a)|_{w'}
\leq
\sum_{w_i \in W(w')} \, \sum_{a \,\in \cal A_{n-1}} \vec v_{n-1}(a) 
\, 
|\sigma_{[0,n-1)}(a)|_{\hat w_i}
+ K_n
\end{equation}
Indeed, any occurrence of $w'$ in $\sigma'_{[0, n)}(a)$ defines in a unique manner an occurrence of some $w_i$ in $\pi_\tau(\sigma_{[0, n-1)}(a))$. The latter defines (again uniquely) an occurrence of $\hat w_i$ in $\sigma_{[0, n-1)}(a)$, unless the corresponding occurrence of $w_i$ in $\pi_\tau(\sigma_{[0, n-1)}(a))$ takes place in a suffix or prefix of length bounded by the maximum $m(w') \geq 0$ of all $|\hat w_i|$. We hence deduce:
$$
K_n \leq 2 m(w') 
\sum_{a \,\in \cal A_{n-1}} \vec v_{n-1}(a) 
$$
It follows now from 
Lemma \ref{2.8e} 
that the right hand side of the last inequality tends to 0 for $n \to \infty$, so that we obtain 
from (\ref{eq3.5e}) and (\ref{eq3.6e}) 
through the above definitions $\mu = \frak m_{\tiny \bvec \sigma}(\bvec v) = \mu^{\tiny \bvec v}$ and $\mu' = 
\frak m_{\tiny \bvec \sigma^\tau}(\bvec v^\tau) = \mu^{\tiny \bvec v^\tau}$ the desired result
$$
\mu^\tau(w') = 
\sum_{w_i \in W(w')} \left( \lim_{n \to \infty} \sum_{a \,\in \cal A_{n-1}} \vec v_{n-1}(a) 
\, 
|\sigma_{[0,n-1)}(a)|_{\hat w_i} \right)$$
$$=
\lim_{n \to \infty}\sum_{a \,\in \cal A_{n-1}} \vec v\,'_n(a) 
\, 
|\sigma'_{[0,n)}(a)|_{w'}
=
\mu'(w')
$$
for any $w' \in \cal B^*$.
\end{proof}

As a first application of the above shown ``basic'' Proposition \ref{3.1d} we derive:

\begin{prop}
\label{4.5}
For any non-erasing morphism $\sigma:\cal A^* \to \cal B^*$ and
any subshift $X \subset \cal A^\Z$ 
with image subshift 
$\sigma(X) $
the induced measure 
transfer 
map 
$$\sigma^\cal M: \cal M(\cal A^\Z) \to \cal M(\cal B^\Z)\, , \,\, \mu \mapsto \mu^\sigma$$
maps the measure cone $\cal M(X)$ surjectively to the measure cone $\cal M(\sigma(X))$.
\end{prop}

\begin{proof}
We consider any everywhere growing directive sequence $\bvec \sigma$ which generates $X$; from Fact \ref{2.1d} we know that such $\bvec \sigma$ exists for any subshift $X$.
By prolonging $\bvec \sigma$ through the morphism $\sigma$ as explained above in Definition-Remark \ref{3.0.5d} we obtain any everywhere growing directive sequence 
$\bvec \sigma' 
:= \bvec \sigma^\sigma$ 
which generates $\sigma(X)$. We then apply Theorem 
\ref{thm1} 
to obtain for any measure $\mu' \in \cal M(\sigma(X))$ a vector tower $\bvec v'$ on $\bvec \sigma'$ with 
$\frak m_{\tiny \bvec \sigma'}(\bvec v') = \mu$. Truncating now the last term of $\bvec v'$ gives a vector tower $\bvec v$ on $\bvec \sigma$, which by 
Remark \ref{2.6e} (2) 
defines a measure $\mu := \frak m_{\tiny \bvec \sigma}(\bvec v)$ on $X$. 
We can now apply Proposition \ref{3.1d} to obtain $\mu' = \mu^\sigma\, [= \sigma^\cal M(\mu)]$.
\end{proof}

\begin{rem}
\label{3.3d}
We'd like to note that as a result of the material presented in this section we have now derived an alternative way how to understand the transferred measure $\mu^\sigma = \sigma^\cal M(\mu) \in \cal M(\cal B^\Z)$, for any non-erasing morphism $\sigma: \cal A^* \to \cal B^*$ and any invariant measure $\mu \in \cal M(\cal A^\Z)$.

It turns out that in many circumstances the use of vector towers as presented here is more convenient 
when dealing with $\mu^\sigma$ in practice, compared to the definition as studied in sections 
3 and 4 of 
\cite{PartI}, and also compared to the approximation method via weighted characteristic measures indicated in 
Remark 3.9 of \cite{PartI}.
\end{rem}

%%%%%%%%%%%%%%%%

\section{Measure towers and vector towers}
\label{sec:4}

Throughout this section we will assume that 
$$\bvec \sigma = (\sigma_n: \cal A_{n+1}^* \to \cal A_n^*)_{n \geq 0}$$
is an everywhere growing directive sequence which generates a subshift $X = X_0 \subset \cal A_0^\Z$
(and where all level maps $\sigma_n$ are non-erasing, see Remark \ref{2.11x}). 
As in (\ref{eq2.11e}) 
we denote 
for any level $k \geq 0$ 
by $X_k \subset \cal A_k^\Z$ the 
intermediate subshift of level $k$, which is 
generated by the truncated sequence 
$
\bvec \sigma\chop_{\! k} = (\sigma_n: \cal A_{n+1}^* \to \cal A_n^*)_{n \geq k}\, $ 
from (\ref{eq2.trunc}).

\begin{defn}
\label{22.x.1}
A {\em measure tower on $\bvec \sigma$}, denoted by $\bvec \mu = (\mu_n)_{n \geq 0}$, is given by a sequence of measures $\mu_n \in \cal M(\cal A_n^\Z)$ which satisfy:
$$\mu_n = \sigma_n^\cal M(\mu_{n+1})$$
The set of measure towers 
on $\bvec \sigma$ will be denoted by $\cal M(\bvec \sigma)$.
\end{defn}

We will now construct a particular type of measure towers on a given directive sequence $\bvec \sigma$ as above, starting from a vector tower $\bvec v = (\vec v_n)_{n \geq 0}$ on $\bvec \sigma$. We first 
observe 
that for any intermediate level $k \geq 0$ we obtain from 
$\bvec \sigma$ 
via the truncated directive sequence 
$
\bvec \sigma\chop_{\! k}$ 
a ``truncated evaluation map'' $\frak m_k := \frak m_{\tiny 
\bvec \sigma\chop_{\! k}}: \cal V(
\bvec \sigma\chop_{\! k}) \to \cal M(X_k)$. From the vector tower $\bvec v$ we obtain similarly a ``truncated'' vector tower $
\bvec v\chop_{\! k} = (\vec v_n)_{n \geq k}$ on $
\bvec \sigma\chop_{\! k}\,$, which defines the corresponding 
shift-invariant ``level $k$ measure''
\begin{equation}
\label{eq4.1d}
\mu_k := \frak m_k(
\bvec v\chop_{\! k})
\end{equation}
on the subshift $X_k \subset \cal A_k^\Z$. From Proposition \ref{3.1d} we obtain directly that $\sigma_{k}^\cal M(\mu_{k+1}) = \mu_k$ for all $k \geq 0$, so that we have:

\begin{lem}
\label{4.2d}
For any vector tower $\bvec v$ 
on 
an everywhere growing 
directive sequence $\bvec \sigma$ the family 
of 
level $k$ 
measures $\mu_k$ as in (\ref{eq4.1d}), for all $k \geq 0$, defines a measure tower 
$\bvec{\frak m}(\bvec v) := (\mu_k)_{k \geq 0}$ on $\bvec \sigma$.
${}^{}$
\qed
\end{lem}

\begin{defn-rem}
\label{back-and-forth}
Conversely, every measure tower $\bvec \mu = (\mu_n)_{n \geq 0}$ on 
a directive sequence $\bvec \sigma$ as above
determines a vector tower $\bvec \zeta(\bvec \mu) = (\vec v_n)_{n \geq 0}$ on $\bvec \sigma$, given by the 
letter frequency vectors $\vec v_n := \vec v(\mu_n) = \zeta_{X_n}(\mu_n)$ from (\ref{eq2.3.1}) and (\ref{eq2.5e}). The fact that $\bvec \zeta(\bvec \mu)$ is indeed a vector tower, i.e. that  
compatibility conditions (\ref{eq2.3.4}) are 
satisfied, is a direct application of 
Proposition 4.5 of \cite{PartI}, stated above as equality (\ref{eq2.3.2}).
\end{defn-rem}

The above set-up of measure towers and vector towers over a given directive sequence is very natural, and indeed, it turns out that the two are essentially equivalent. More precisely, we obtain:

\begin{prop}
\label{22.9}
For any everywhere growing directive sequence $\bvec \sigma$ 
there is a canonical 
$\R_{\geq 0}$-linear bijection 
$$\bvec \zeta: \cal M(\bvec \sigma) \to \cal V(\bvec \sigma)$$
between the 
cone of measure towers on one hand and the 
cone of vector towers on the other, given by the map
$$
\bvec \mu \mapsto \bvec \zeta(\bvec \mu) \qquad\text{and its inverse} \qquad 
\bvec v \mapsto \bvec{\frak m}(\bvec v) \, .$$
\end{prop}

\begin{proof}
The fact, that the composition $\bvec \zeta \circ \bvec{\frak m}
$ gives the identity on $\cal V(\bvec \sigma)$, 
follows directly from  Proposition \ref{2.4d} (1), 
when applied to all truncated sequences $\bvec \sigma\chop_{\! k}$ with $k \geq 0$. We obtain in particular that the map $\bvec{\frak m}$ is injective.

On the other hand, we can apply Theorem \ref{thm1} to each of the truncated sequences $
\bvec \sigma\chop_{\! k}$ 
to obtain the surjectivity of the 
map $\frak m_k: \cal V(
\bvec \sigma\chop_{\! k}) \to \cal M(X_k)$ for any level $k \geq 0$. 
It follows then directly from the definition set up in Lemma \ref{4.2d} above that the map $\bvec{\frak m}: \cal V(\bvec \sigma) \mapsto \cal M(\bvec \sigma)$ must be surjective.

Hence $\bvec{\frak m}$ is a bijective map, which implies that $\bvec \zeta$ must also be bijective, and that $\bvec{\frak m} \circ \bvec \zeta$ is the identity on $\cal M(\bvec \sigma)$.

The linearity of the maps $\bvec \zeta$ and $\bvec{\frak m}$ is a direct consequence of the linearity (see section \ref{sec:2.3}) of the maps $\frak m_k$ and $\zeta_{X_n}$ used in the above definitions of the measure 
or vector 
towers $\bvec{\frak m}(\bvec v) = (\frak m_k(
\bvec v\chop_{\! k}))_{k \geq 0}$ and $\bvec \zeta(\bvec \mu) = (\zeta_{X_n}(\mu_n))_{n \geq 0}$ respectively.
\end{proof}

Although a bit similar in notation, the two cones $\cal M(\bvec \sigma)$ and $\cal M(X_{\tiny\bvec \sigma})$ should not be confused. Indeed, 
without further assumptions on the given set-up, the structure of the cone $\cal M(\bvec \sigma)$ of measure towers will not only depend on the given subshift $X = X_{\tiny \bvec \sigma}$ but can vary quite a bit depending on the choice of the $S$-adic development $\bvec \sigma$ of $X$. 
More precisely, we have:

\begin{rem}
\label{4.5.5d}
For any everywhere growing directive sequence $\bvec \sigma$ which generates a subshift $X = X_{\tiny \bvec \sigma}$ the composition
\begin{equation}
\label{eq4.2e}
\frak m_{\tiny \bvec \sigma} \circ \bvec \zeta : \cal M(\bvec \sigma) \to \cal M(X)
\end{equation}
is linear and surjective 
since $\bvec \zeta$ is linear and bijective by Proposition \ref{22.9}, and $\frak m_{\tiny \bvec \sigma}$ is linear and surjective by Theorem \ref{thm1}. 
However, 
in general the map $\frak m_{\tiny \bvec \sigma} \circ \bvec \zeta$ will be far from being injective.
\end{rem}

We thus consider the following 
strengthening on the hypotheses of the given directive sequence, which has been considered already by several other authors in a related context (compare Definition 4.1 of \cite{BSTY19} or subsection 3.3 of 
\cite{DDMP}):

\begin{defn} 
\label{4.5d}
A directive sequence (or an $S$-adic development) $\bvec \sigma = (\sigma_n)_{n \geq 0}$ 
is called {\em totally recognizable} if every level map $\sigma_n$ 
is recognizable 
in the corresponding subshift $X_{n+1}$
(see Definition \ref{6.1} and Proposition \ref{recognizable-e}). 
If all but finitely many of the level maps $\sigma_n$ are 
recognizable 
in $X_{n+1}$, we call $\bvec \sigma$ {\em eventually recognizable}.
\end{defn}

\begin{thm}
\label{4.6d}
For any everywhere growing totally recognizable $S$-adic development $\bvec \sigma$ of a subshift $X$ and its associated cone $\cal V(\bvec \sigma)$ of vector towers the canonical linear map
$$\frak 
m_{\tiny \bvec \sigma}: \cal V(\bvec \sigma) \to \cal M(X)$$
is a bijection.

In particular, for any level $n \geq 0$ the map $\sigma_{[0, n)}^\cal M: \cal M(X_n) \to \cal M(X)$ is a 
linear bijection of cones.
Similarly, the same conclusion follows for the map 
$\frak m_{\tiny \bvec \sigma} \circ \bvec \zeta$ from (\ref{eq4.2e}).
\end{thm}

\begin{proof}
From the assumption that $\bvec \sigma$ is totally recognizable it follows (using statement (3d) of Theorem \ref{2.2.2}) 
that the induced linear map 
$$(\sigma_n)^\cal M_{X_{n+1}}: \cal M(X_{n+1}) \to \cal M(X_{n})$$
is bijective for any level $n \geq 0$. It follows that the composed map 
$m_{\tiny \bvec \sigma}
\circ \bvec \zeta: \cal M(\bvec \sigma) \to \cal V(\bvec \sigma) \to  \cal M(X)$ from (\ref{eq4.2e}) is bijective. Since we know from Proposition \ref{22.9} that the map $\bvec \zeta$ is a bijection, 
we deduce that 
$m_{\tiny \bvec \sigma}$ 
must be bijective.
\end{proof}

%%%%%%%%%%%

\section{Directive sequences with ``small'' intermediate letter frequency cones}
\label{sec:left-over-2.8d}

In this section we will give 
a first 
application of the machinery set up in the previous two sections. But  before doing so we want to summarize, for the convenience of the reader, the various ingredients that the rich picture issuing from this set-up offers, and to list some basic facts in order to avoid potential misunderstandings. As an illustration, we give at the end of this section a detailed example, where all the data listed now can be seen in practice.

\smallskip

We use the same terminology as previously, i.e. 
$X \in \cal A^\Z$ is a subshift over the finite alphabet $\cal A = \cal A_0$, and $\bvec \sigma = (\sigma_n: \cal A_{n+1}^* \to \cal A_n^*)_{n \geq 0}$ is an everywhere growing directive sequence which generates $X$. We are 
particularly interested in the {\em intermediate letter frequency cones} $\cal C_n = \cal C_n(\bvec \sigma) \subset \R_{\geq 0}^{\cal A_n}$ and in particular in their dimension
$$c_n := \dim \cal C_n \leq \card(\cal A_n)\, .$$
The cone $\cal C_n$ is given as the image of the set $\cal V(\bvec \sigma)$ of vector towers under the level $n$ projection map $pr_n\,$, which amounts to stating that $\cal C_n$ is the intersection of the nested images of the non-negative alphabet cones of level $m \geq n$ under the telescoped level maps, i.e.
$$\cal C_n = \bigcap\,
\{
\R_{\geq 0}^{\cal A_n}\supseteq \ldots \supseteq M(\sigma_{[n, m)})(\R_{\geq 0}^{\cal A_{m+1}}) \supseteq \ldots\}$$
In particular one always has
\begin{equation}
\label{eq6.0e}
\cal C_n = M(\sigma_n) ( \cal C_{n+1})
\end{equation}
and thus 
$$c_n \leq c_{n+1}$$
for all $n \geq 0$.

An 
alternative interpretation of the intermediate level frequency cones is given through the level subshifts $X_n$ (defined by the truncated directive sequences $\bvec \sigma\chop_n$) and their measure cones $\cal M(X_n)$, which, when evaluated via the associated letter frequency vectors, results into
$$\cal C_n = \zeta_{X_n}(\cal M(X_n))\, ,$$
where $\zeta_{X_n}: \cal M(X_n) \to \R_{\geq 0}^{\cal A_n}$ is given for $\cal A_n = \{a_{n, 1}, \ldots, a_{n, d(n)}\}$ by $\mu \mapsto ([\mu(a_{n, 1})], \ldots, [\mu(a_{n, d(n)})])$ for any $\mu \in \cal M(X_n)$.

\smallskip

Our main focus here is to explain how this set-up and in particular the value of the $c_n$ can be used in order to find out information about the number $e(X) \in \N \cup \{\infty\}$ of invariant ergodic probability measures on $X$. 

\begin{rem}
\label{5.0.n}
Under the above stated conditions the following conclusions are immediate:
\begin{enumerate}
\item
It is quite possible that $e(X) > c_n$ for some ``low'' level $n \geq 0$, even if $\bvec \sigma$ is totally recognizable.
\item
The converse inequality, $e(X) < c_n$, is also possible, but in this case the directive sequence $\bvec \sigma$ is not totally recognizable. More precisely, in this case the telescoped morphism $\sigma_{[0, n)}$ is not recognizable.
\item
In any case, we always have 
$$e(X) \,\,\leq\,\, \lim c_n \,\,\leq\,\, \liminf (\card\,\cal A_n) \, ,$$
but in general both inequalities may well be strict.
\item
However, if $\bvec \sigma$ is totally recognizable, then we have
$$e(X) = \lim c_n$$
In particular, we recover the well known upper bound $e(X) \leq \liminf (\card\,\cal A_n)$, as well as the lower bounds $c_n \leq e(X)$ for all $n \geq 0$.
\end{enumerate}
\end{rem}

\medskip

From 
Remark \ref{5.0.n} (3) 
we observe directly that for any directive sequence $\bvec \sigma$ with 
finite alphabet rank 
(i.e. $\liminf(\card\ \cal A_n) < \infty$) 
there is a {\em critical level} $n_0 \geq 0$ such that one has 
\begin{equation}
\label{eq6.2e}
\text{$c_n = c_{n_0}\,\,$ for all $\,\,n \geq n_0\,\,$ and $\,\,c_n < c_{n_0}\,\,$ for all $\,\,n < n_0$\,.} 
\end{equation}
More generally, any everywhere growing directive sequence $\bvec \sigma$ (possibly with infinite alphabet rank) which possesses such a critical level has been termed in \cite{BHL1} {\em thinning}, and in the particular case where the critical level agrees with the base level $n_0 = 0$, the sequence $\bvec \sigma$ has been called {\em thin}. Of course, any thinning sequence can be made thin by simply truncating it at its critical level (or any level higher up); furthermore, we can telescope all levels below the critical level into a single ``thinning'' morphism. Subshifts that are ``thin'' in that they are generated by a thin (and 
in particular 
everywhere growing) directive sequence have the following useful property:

\begin{prop}[\cite{BHL1}]
\label{6.1d}
Let $X \subset \cal A^\Z$ be a subshift generated by a thin directive sequence $\bvec \sigma$. Then the letter frequency 
map $\zeta_X: \cal M(X) \to \R_{\geq 0}^\cal A$ co-restricts to a linear bijection
$$
\cal M(X) \to C(X) \, , \,\, \mu \mapsto (\mu(a_k))_{a_k \in \cal A} \, .
$$
In particular, 
any two invariant measures $\mu_1$ and $\mu_2$ on $X$ are equal if and only if one has $\mu_1([a_k]) = \mu_2([a_k])$ for the finitely many cylinders $[a_k]$ given by all letters $a_k \in \cal A$.
\end{prop}

This statement 
can be derived directly from
Proposition 
10.2 (1) 
and Corollary 10.4 of \cite{BHL1}. 
For convenience of the reader we 
give here a proof in 
the terminology introduced above.

\begin{proof} [Proof of Proposition \ref{6.1d}]
For any two measure $\mu, \mu' \in \cal M(X)$ there exist, by Theorem \ref{thm1}, vector towers $\bvec v = (\vec v_n)_{n \geq 0}$ and $\bvec v' = (\vec v\, '_{\! n})_{n \geq 0}$ on $\bvec \sigma$ with $\frak m_{\tiny \bvec \sigma}(\bvec v) = \mu$ and $\frak m_{\tiny \bvec \sigma}(\bvec v') = \mu'$. Thus $\mu \neq \mu'$ implies $\bvec v \neq \bvec v'$ and hence $\vec v_n \neq \vec v\, '_{\! n}$ for some $n \geq 0$. But then we deduce from (\ref{eq6.0e}) and the hypothesis that $\dim \cal C(X_n) = c_n = c_0 = \dim \cal C(X)$ that $\vec v_0 = M(\sigma_{[0, n)})(\vec v_n) \neq M(\sigma_{[0, n)})(\vec v\, '_{\! n}) = \vec v\, '_{\! 0}$. From Proposition \ref{2.4d} (1) we know that $\vec v_0 = pr_0(\bvec v) = \zeta_X(\mu)$ and $\vec v\, '_{\! 0} = pr_0(\bvec v') = \zeta_X(\mu')$, which shows that the map $\zeta_X$ is injective. For the linearity of $\zeta_X$ and the equality $\zeta_X(\cal M(X)) = \cal C(X)$ see (\ref{eq2.5e}) and (\ref{eq2.3.3}).
\end{proof}

Directive sequences of 
finite 
alphabet rank 
occur naturally in many important contexts 
in the symbolic dynamics literature 
(e.g. substitutive subshifts, IETs, etc). Furthermore, the extra invertibility condition from the following proposition is rather frequently 
satisfied.

\begin{cor}
\label{6.2d}
Let $X \subset \cal A^\Z$ be a subshift generated by an everywhere growing directive sequence $\bvec \sigma = (\sigma_n)_{n \geq 0}$ of finite alphabet rank. 
Assume that 
for every $n \geq 0$ the incidence matrix $M(\sigma_n)$ is invertible over $\R$. 
Then any invariant measure $\mu$ on the subshift $X$ 
is determined by the evaluation of $\mu$ on the letter cylinders, i.e. by the values $\mu([a_k])$ for all $a_k \in \cal A$.
\end {cor}

\begin{proof}
From (\ref{eq6.0e}) and the hypothesis that $M(\sigma_n)$ is invertible it follows directly 
that $c_{n+1}= c_n$ for all $n \geq 0$, so that 
the directive sequence $\bvec \sigma$ is thin. Hence the 
hypotheses of Proposition \ref{6.1d} are satisfied, which gives directly the claimed statement.
\end{proof}

Note that the conclusion of Corollary \ref{6.2d} has recently been proved by Berth\'e et al. 
under 
somewhat 
more restrictive hypotheses 
(see Corollary 4.2 of \cite{B+Co}); in particular it is required there that every $M(\sigma_n)$ has determinant equal to $1$ or to $-1$, and that $X$ is minimal.

\begin{rem}
\label{5.4.n}
(1)
If in Proposition \ref{6.1d} the hypothesis ``thin'' is replaced by ``thinning'', with critical level $n_0 \geq 1$, then the conclusion, that any two distinct measures $\mu \neq \mu' \in \cal M(X)$ can be distinguished by the evaluation on the letter cylinders $[a_k]$ for all $a_k \in \cal A$, may in some cases still hold, despite the fact that from the definition of the critical level we have 
$$\dim \cal C_0 = c_0 < c_{n_0} = \dim \cal C_{n_0} = \dim \cal M(X_{n_0})\, .$$
Here the last equality follows from Proposition \ref{6.1d}, applied to the directive sequence truncated at the critical level $n_0$.
The reason, why the above strict inequality doesn't contradict the presumed equality $c_0 = \dim \cal C_0 = \dim \cal M(X)$, is that the measure transfer map $\sigma_{[0, n_0)}^\cal M: \cal M(X_{n_0}) \to \cal M(X)$ may well not be injective, in the case that the telescoped level map $\sigma_{[0, n_0)}$ is not recognizable in the level subshift $X_{n_0}\,$. 

However, if $\bvec \sigma$ is totally recognizable, 
or if at least $\sigma_{[0, n_0)}$ is recognizable in $X_{n_0}$, and if furthermore $\bvec \sigma$ is thinning but not thin, then the conclusion of Proposition \ref{6.1d} necessarily fails: this case is treated in Example \ref{6.3d} below.

\smallskip
\noindent
(2)
In view of the fact that the measure transfer map $\sigma^\cal M$ induced by a non-recognizable monoid morphisms $\sigma$ is in general far from being injective, it seems noteworthy that in Proposition \ref{6.1d} and Corollary \ref{6.2d} no recognizability condition on the level maps $\sigma_n$ is imposed. One should recall in this context that in \cite{BSTY19}, Theorem 5.2, it has been proved that directive sequences of bounded alphabet rank, with aperiodic level subshifts, are ``eventually recognizable'', i.e. all level maps above some ``other critical level'' must be recognizable in their level subshift. But this ``other critical level'' may well be a lot bigger than the above critical level $n_0$, and indeed we give in 
Corollary \ref{2-interesting-sequences} (2) examples of thin directive sequences where this ``other critical level'' can be chosen to be arbitrarily high up, while none of the level morphisms below it is recognizable in its corresponding level subshift
(which is aperiodic for any level).
\end{rem}

\smallskip

We 
now 
present the promised ``detailed example with all above data made visible'':

\begin{example}
\label{6.3d}
The subshift $X$ in this example consists of two periodic words and is hence all by itself not so interesting. We chose it in order to give a 
transparent 
presentation of 
a simple subshift 
via some not so obvious everywhere growing directive sequence, which we now describe in detail. We first describe the level $n = 1$, then pass to the base level $n = 0$, and finally built the higher levels $n \geq 2$ on top of the two lowest levels. We also include for each level $n$ a description of the measure cone $\cal M(X_n)$ and of the associated letter frequency cone $\cal C_n$.

\smallskip

Set $\cal A_1 = \{a, b\}$, and let $X_1 \subset \cal A_1^\Z$ be the union of the two periodic subshifts 
$\cal O(w^{\pm \infty})$ and $\cal O({w'}^{\pm \infty})$, 
defined by the words $w = a^2 b$ and $w' = b^2 a$. We consider the two characteristic measures 
$\mu := \mu_{w}$ and $\mu' := \mu_{w'}\,$, 
and observe that $\cal M(X_1)$ consists of all non-negative linear combinations of these two measures. The letter frequency map $\zeta_{X_1}: \cal M(X_1) \to \cal C(X_1) \subset \R_{\geq 0}^{\{a, b\}}$ is injective, in that $\zeta_{X_1}(\mu) = 2 \vec e_a + \vec e_b$ and $\zeta_{X_1}(\mu') = \vec e_a + 2 \vec e_b$. This results into $c_1 = \dim(\cal C_1) = 2$.

\smallskip

For $\cal A_0 = \{c, d\}$ consider now the ``Thue-Morse'' morphism $\sigma_0
: \cal A_1^\Z \to \cal A_0^\Z\, , \,\, a \mapsto cd\, , \,\, b \mapsto dc$, and recall (see Proposition \ref{2.2.3} (d)) 
that $\sigma_0^\cal M(\mu) = \mu_{\sigma_0(w)}$ and  $\sigma_0^\cal M(\mu') = \mu_{\sigma_0(w')}\,$, with $\sigma_0(w) = cdcddc$ and $\sigma_0(w') = dcdccd$. 
Since $cdcddc$ and $dcdccd$ 
can not be obtained from each other by a cyclic permutation, we have $\cal O((cdcddc)^{\pm \infty}) \neq \cal O((dcdccd)^{\pm \infty})$, so that from (\ref{last-minute}) it follows that $\supp(\mu_{cdcddc}) \neq \supp(\mu_{dcdccd})$. We thus deduce 
for the image subshift $X_0 = \sigma_0(X_1)$ that the measure cone $\cal M(X_0)$, which is spanned by $\mu_{cdcddc}$ and $\mu_{dcdccd}\,$, is of dimension~$2$.

On the other hand, 
using Proposition \ref{2.8.formula} (or more directly, equality (\ref{eq:charac-m}))
we readily compute $\mu_{cdcddc}([cd]) = \mu_{cdcddc}([dc]) = 
2$ as well as $\mu_{dcdccd}([cd]) = \mu_{dcdccd}([dc]) = 
2$.
It follows that the frequency map $\zeta = \zeta_{X_0}$ is not injective, and that 
$\cal C_0$ has dimension $c_0 = 1$.

\smallskip

We now define the higher up levels of the directive sequence by setting $\cal A_n = \{x, y\}$ for any $n \geq 2$, and by defining all level morphisms $\sigma_n: \cal A_{n+1} \to \cal A_n$ for $n \geq 2$ to be equal to the substitution defined by $x \mapsto x^2$, $y \mapsto y^2$. It follows that 
for $n \geq 2$ all level subshifts $X_n$ consist of the two biinfinite periodic words $x^{\pm \infty}$ and $y^{\pm \infty}$. 
Moreover, we easily see that the incidence matrix of $\sigma_n$ is equal to 2 times the 2-by-2 identity matrix, $M(\sigma_n) = 2 \cdot I_2$, so that we have $\cal M(X_n) = \cal C_n = \R_{\geq 0}^{\{x, y\}}$.

\smallskip

It remains now to define $\sigma_1: \cal A_2 \to \cal A_1$ via $x \mapsto w$, $y \mapsto w'$, 
which ensures $\sigma_1(X_2) = X_1$, in order 
to obtain a directive sequence $\bvec \sigma = (\sigma_n)_{n \geq 0}$ over alphabets that all have cardinality 2. We have shown above that the critical level of this directive sequence is $n_0 = 1$, while the 
evaluation on the 
cylinders $[\sigma_0(a)] = [cd]$ and  $[\sigma_0(b)] = [dc]$ 
does not suffice to distinguish the two measures 
$\mu_{cdcddc} \neq \mu_{dcdccd}$
that span $\cal M(X_0)$.
\end{example}

%%%%%%%%%%%%%%%%%%%%%%
\section{Minimal subshifts with zero entropy and infinitely many ergodic probability measures}
\label{sec:Cyr-Kra}

A subshift $X$, which is ``small'' in that it has topological entropy  $h_X = 0$
(see (\ref{eq2.ent})), and simultaneously ``large'' in that the number $e(X)$ of ergodic probability measures carried by $X$ is infinite, is a bit of a contradiction in itself
(if one restricts to non-atomic measures). However, such subshifts are known to exist, but they are not easy to come by. 
One of the first such subshift known to us is the Pascal-adic subshift, treated in \cite{MP}; more recent such examples (with additional strong properties,
in particular minimality) are exhibited in \cite{CK2}. Not surprisingly, there is always a certain amount of labor involved 
in order to 
simultaneously 
achieve the above two 
opposite properties.

\smallskip

In this section we will 
present 
an alternative way to construct 
minimal subshifts $X$ which satisfy both,  $h_X = 0$ and $e(X) = \infty$. 
The main purpose of this section is to underline  
how directly such examples can be 
exhibited 
by means of the technology established in the previous sections.

\smallskip

We first 
recall 
two known results. The first appears as Theorem 4.3 in \cite{BD} and is attributed there to Thierry Monteil; alternatively it can be found in \cite{CANT} as Lemma 6.7.1 of Chapter 6, written by Fabien Durand, who told us that 
the result can actually be traced back to the paper \cite{BH} by Boyle-Handelman. 

\begin{prop}
\label{Monteil-bound}
Let $X$ be 
a subshift 
which is generated by a directive sequence $\bvec \sigma = (\sigma_n)_{n \geq 0}$ 
with level alphabets $\cal A_n$. Then, for the minimal level letter image length $\beta_-(n)$ from (\ref{eq2.9e}), one has:
$$h_X \leq \inf_{n \geq 0} \frac{\log (\text{\rm \card} \,\cal A_n)}{\beta_-(n)}$$
\end{prop}

\begin{prop}[Section 4.1 of \cite{BHL2}]
\label{5.3d}
For any integer $d \geq 2$ 
let $X$ be 
a subshift 
which is generated by a 
directive sequence 
${\bvec \sigma} = (\sigma_{n,d})_{n \geq 0}$ 
with level alphabets 
that are all 
of 
uniform cardinality $d$
(and are thus identified with $\cal A_{(d)} = \{a_1, \ldots, a_d\}$). Assume that for any level $n \geq 0$ the incidence matrix of the 
level map $\sigma_{n,d}$ is given by
\begin{equation}
\label{eq5.1e}
M(\sigma_{n,d}) = M_{\ell(n), d} :=  \ell(n) I_d + 1_{d \times d}\, ,
\end{equation}
where $I_d$ is the identity matrix of size $d \times d$, $1_{d \times d}$ is the $d \times d$-matrix with all entries equal to 1, and $\ell(n)$ is a positive integer depending on $n$. 

Then $X$ is minimal, and for any sufficiently fast growing sequence $(\ell(n))_{n \in \N}$ the subshift $X$ admits precisely $d$ distinct invariant ergodic probability measures.
\end{prop}

The use of Proposition \ref{Monteil-bound} will be an ingredient in our proof below.
Proposition \ref{5.3d}, on the other hand, will not be formally used in the sequel, 
but it may pay anyway for the reader to look it up: We use below the very same basic idea as in this earlier result, but do not carry out all calculations as had been done in section 4 of \cite{BHL2} (where in particular precise lower bounds for the 
integers 
$\ell(n)$ are computed 
which 
guarantee the ``sufficiently fast growing'' in the above statement).

\medskip

For the proof below we first need to define 
for any integer $d \geq 2$ 
and 
alphabet 
$\cal A_{(d)} = \{a_1, \ldots, a_d\}$ 
the 
morphism $\tau_{d}: \cal A_{({d+1})}^* \to \cal A_{(d)}^*\,$, given by $a_i \mapsto a_i^2$ for any $a_i$ with $1 \leq i \leq d$ and $a_{d+1} \mapsto a_1 a_2\ldots a_d$. 

\begin{rem}
\label{recog}
(1)
For the morphism $\tau_d$ 
as given above it 
is easy to see that any biinfinite word 
${\bf y} \in \cal A_{(d)}^\Z \smallsetminus \{a_1^{\pm \infty}, a_2^{\pm \infty}, \ldots, a_d^{\pm \infty}\}$ 
can be 
``desubstituted'' in 
at most 
one way 
(compare Remark 
6.2 (2) of \cite{PartI}) 
to give a biinfinite word ${\bf x} \in \cal A_{({d+1})}^\Z$ with 
$\tau_d(\cal O({\bf x})) = \cal O({\bf y})$. 
Since for any $i = 1, \ldots, n$ the periodic word $a_i^{\pm\infty}$ is the only element ${\bf x} \in \cal A_{(d+1)}^\Z$ with $\tau_d(\cal O({\bf x})) = \cal O(a_i^{\pm\infty})$, it follows
that $\tau_d$ is 
recognizable in every subshift which doesn't contain any of the periodic words $a_i^{\pm \infty}$.

\smallskip
\noindent
(2)
Again 
by elementary desubstitution arguments one verifies quickly 
that any morphism $\sigma_{n,d}$ with incidence matrix given by equality (\ref{eq5.1e}), 
with $\ell(n) \geq 2$, 
is 
recognizable 
in the full shift $\cal A_{(d)}^\Z$.

[Indeed, it suffices to check in any biinfinite word ${\bf y} 
\in \sigma_{n,d}(\cal A_{(d)}^\Z)$ for a factor $w
\in \cal A_{(d)}^*$ which is ``distinguished'' in that some letter $a_i \in \cal A_{(d)}$ occurs precisely 3 times in $w$, while all other letters $a_j \in \cal A_{(d)}$ occur at most twice. 
Such a distinguished word $w$ occurs in 
$\sigma_{n,d}(a_i)$, 
and any such occurrence is contained in the image of some word from $\cal A_{(d)}^*$ of length at most 3. In either case one verifies quickly that the middle occurrence of $a_i$ in $w$ must belong to $\sigma_{n,d}(a_i)$. For this middle occurrence $y_s$ in the factor $w = y_r \ldots y_t$ of ${\bf y}  = \ldots y_{n-1} y_{n} y_{n+1} \ldots$ one considers the factors $w_+ = y_s \ldots y_{t'}$ and $w_- = y_{r'} \ldots y_s$ of ${\bf y}$, with $y_s = y_{s+1} = \ldots = y_{t'-1} = a_i$ and $y_{t'} \neq a_i$, and similarly $y_s = y_{s-1} = \ldots = y_{r'+1} = a_i$ and $y_{r'} \neq a_i$. From the fact that $\sigma_{n,d}(a_i)$ contains each letter $a_j \neq a_i$ precisely once one deduces directly that the words $w_+$ and $w_-$ determine which occurrence of $a_i$ in $\sigma_{n,d}(a_i)$ is given by the letter $y_s$. It follows that, starting from $y_s\,$, the biininite word ${\bf y}$ can be desubstituted in precisely one way.]

\smallskip
\noindent
(3)
From the conditions on the incidence matrix $M(\sigma_n)$ in (\ref{eq5.1e}) it follows directly that every word in $\sigma_{n, d}(\cal A_{(d)}^*)$ must contain each of the letters of $\cal A_{(d)}$. 
Hence we observe 
that $\sigma_{n, d}(\cal A_{(d)}^\Z)$ can not contain any of the periodic words $a_i^{\pm \infty}$.

\smallskip
\noindent
(4)
As a consequence of the above observations (1) - (3) we deduce for the alternating directive sequence
\begin{equation}
\label{eq5.1d}
\bvec \sigma = 
\sigma_{2} \circ \tau_2 \circ \sigma_{3} \circ \tau_3 \circ \ldots \, ,
\end{equation}
where we set $\sigma_d := \sigma_{d, d}\,$, that each level map is recognizable in its corresponding level subshift, so that the sequence $\bvec \sigma$ is fully recognizable.
\end{rem}

\begin{thm}
\label{9.2x}
For any integer $d \geq 2$ let 
$\cal A_{(d)} = \{a_1, \ldots, a_d\}$ and let 
$\sigma_d: \cal A_{(d)}^* \to \cal A_{(d)}^*$ be a morphism with incidence matrix 
$M(\sigma_d) = M_{\ell(d), d}$
from (\ref{eq5.1e}), for some integer 
$\ell(d) \geq 2$
depending on $d$. 
Let $X$ be the subshift generated by the alternating directive sequence 
$\bvec \sigma$ given in (\ref{eq5.1d}).

If the exponent sequence $(\ell(n))_{n \in \N}$ is sufficiently fast growing, then the subshift $X$ is minimal, has entropy $h_X = 0$ and admits infinitely many distinct invariant ergodic probability measures.

(We denote by $\frak X$ the class of all subshifts $X \subset \cal A_{(d)}^\Z$ which satisfy all of the above conditions.)
\end{thm}

\begin{proof}
For each 
integer $d \geq 2$ 
we identify the 
finite alphabet $\cal A_{(d)} = \{a_1, a_2, \ldots, a_d\}$ with the corresponding subset of an infinite alphabet, via $\cal A_{(2)} \subset \cal A_{(3)} \subset \ldots \subset \cal A_{(\infty)} = \{a_1, a_2, \ldots\}$. For the issuing infinite non-negative cone $\R_{\geq 0}^{\cal A_{(\infty)}}$ we abbreviate for notational convenience the base unit vectors to $\vec e_i := \vec e_{a_i}$.

For any level $n = 2d -2$ or $n = 2d - 1$ we consider the 
subcone $\cal C^n := \R_{\geq 0}^{\cal A_{(d)}} \subset \R_{\geq 0}^{\cal A_{(\infty)}}$, 
and in particular 
the ``center vector'' $\vec c_n = \sum \vec e_i$ of $\cal C^n$. We observe that both families, the morphisms $\sigma_d$ as well as the morphisms $\tau_d\,$, induce maps $M(\sigma_d): \cal C^n \to \cal C^n$ and $M(\tau_d): \cal C^{n+1} \to \cal C^n$ 
respectively 
which each 
maps 
the center vector $\vec c_{n}$ (for $\sigma_n$) or $\vec c_{n+1}$ (for $\tau_n$) to a scalar multiple of the center vector $\vec c_n\,$. 
Furthermore, any unit vector $\vec e_i$ with $1 \leq i \leq d$ 
is mapped by both, $M(\sigma_d)$ and $M(\tau_d)$, 
to a non-negative linear combination $\lambda_1 \vec e_i + \lambda_2 \vec c_d$. Note here that 
(again for both, $\sigma_d$ and $\tau_d$,) 
\begin{equation}
\label{eq5.x}
\text{the quotient $\frac{\lambda_2}{\lambda_1}$ can be made arbitrarily small}
\end{equation}
by choosing $\ell(d)$ sufficiently large.

We now fix some level $n_0 = 2d-2 \geq 0$, and for any index $i$ with  
$1 \leq i \leq d$ we look for a vector tower $\bvec v_{\! i} = (\vec v^{\, i}_{n})_{n \geq n_0}$ on the truncated directed sequence $
\bvec \sigma\chop_{\! n_0} = \sigma_{n_0} \circ \tau_{n_0} \circ \sigma_{n_0+1} \circ \tau_{n_0+1} \circ \ldots$ 
with the property that $\bvec v_{\! i}$ 
has for any level $n \geq n_0$ a level vector $\vec v^{\, i}_n = \lambda_{1, n} \vec e_i + \lambda_{2,n} \vec c_n$, with coefficients 
\begin{equation}
\label{eq5.3e}
\lambda_{1, n} > 0 \qquad \text{and}  \qquad \lambda_{2, n} > 0
\end{equation}
(which must  
both tend to 0 for $n \to \infty$).
From (\ref{eq5.x}) we deduce that 
a sufficiently large choice of the exponents $\ell(d)$ 
effects indeed that there exist families of such coefficients 
where 
both of the inequalities in (\ref{eq5.3e}) are satisfied, while the compatibility condition (\ref{eq2.3.4}) is maintained, for any $n \geq n_0$. It follows that  
on the lowest level $n = n_0$ (and thus similarly also on all 
levels 
$n \geq n_0$) the level vectors
$\vec v^{\, 1}_{n_0}, \vec v^{\, 2}_{n_0}, \ldots, \vec v^{\, d}_{n_0}$ are linearly independent.

For the level subshift 
$X_{n_0} \subset \cal A_{(d)}^\Z\,$, 
generated by the truncated sequence 
$\bvec \sigma\chop_{\! n_0}\,$, 
the 
truncated evaluation map 
$\frak m_{n_0} := \frak m_{\tiny \bvec \sigma\chop_{\! n_0}}\!\!\!\!: \bvec{\cal V}(\bvec \sigma\chop_{\! n_0}) \to \cal M(X_{n_0})$ 
from 
(\ref{eq4.1d}) 
defines $d$ invariant measures $\mu_1, \ldots, \mu_d$ 
on the level subshift $X_{n_0}$ 
as images of the $d$ vector towers $\bvec v_{\! 1}, \ldots, \bvec v_{\! d}$ respectively:
$$\mu_i \,\, = \,\, \frak m_{n_0}(\bvec v_{\! i})$$
It follows from Proposition \ref{2.4d} (1) that 
the subcone 
$$\cal M_{n_0}
:= {\R_{\geq 0}}\, \langle \mu_1, \ldots, \mu_d\rangle \subset 
\cal M(X_{n_0})$$
spanned by the $\mu_i$ has dimension $d$.
Since 
we verified 
in Remark \ref{recog} (4) above 
that each of the maps $\sigma_j$ and $\tau_j$ 
is 
recognizable in its corresponding level subshift, 
it follows from Theorem \ref{2.2.2} 
(3d) that $\cal M_{n_0}$ is mapped by $\sigma_2^\cal M \circ \tau_2^\cal M \circ \ldots \circ \sigma_{n_0 -1}^\cal M \circ \tau_{n_0 -1}^\cal M$ 
to a subcone of $\cal M(X)$ that also has dimension $d$.

We have thus proved that $\cal M(X)$ contains subcones of arbitrary large dimension, and hence must be infinite dimensional, 
i.e. $e(X) = \infty$. 
The desired equality $h_X =0$ is immediate from Proposition \ref{Monteil-bound} for large $\ell(d)$, and the minimality of $X$ follows 
directly from the positivity of the matrices $M(\sigma_d)$, see Remark \ref{2.2d} (4).
\end{proof}

%%%%%%
\section{Non-recognizable directive sequences}
\label{sec:pseudo-A}

The purpose of this section is to show how non-recognizable morphisms appear naturally in a well known context (IETs and pseudo-Anosov surface homeomorphisms), and how this phenomenon can be exploited to construct interesting directive sequences that are not totally recognizable or even not eventually recognizable.

\smallskip

Our construction will be presented in 4 steps, organized below 
as follows: In subsection \ref{sec:8.1} we present our basic quotient construction in geometric language. In subsection \ref{sec:8.2} we show how the canonical ``inverse quotient construction'' is obtained in a natural geometric context, to define a non-recognizable monoid morphism. In subsection \ref{sec:8.3} the results from the previous subsections are properly ``pasted together'' to give a directive sequence where every level morphism is non-recognizable (and in addition it is a particularly nice letter-to-letter factor map). Finally, in subsection \ref{sec:8.4} we modify this sequence slightly to obtain the desired everywhere growing but not (eventually) recognizable directive sequences. Note that all intermediate level subshifts which occur in our constructions turn out to be minimal; they are furthermore both, substitutive and IET.

Before starting the detailed description, we will highlight its essential features in a special case, in a language that may be more easily accessible to those of us who are less familiar with Thurston's work on surface homeomorphisms:

\begin{rem}
\label{8.0d}
(1)
Let us consider the tiling of the real plane $\R^2$ by squares of side length 1 that have their vertices on the points with integer coordinates. We now pick a slope $s$, say $0 < s < 1$, and we foliate the plane by lines that have slope $s$. 
By choosing the slope $s$ to be irrational, we make sure that 
on any line of the foliation 
there is at most one 
vertex of our square tiling. 
To every line $\ell$ that avoids 
any such vertex 
one can associate canonically a biinfinite word $w(\ell)$ in the letters $h$ and $v$, which records the sequence of intersections of $\ell$ with a horizontal (``$h$'') or vertical (``$v$'') line of our square grid. In order to fix an indexing of the letters of $w(\ell)$ we pick a distinguished ``base square'' $Q$ and require that $\ell$ passes through the interior of $Q$. 
We quickly observe 
that the orbits in 
our family of lines $\ell$, with respect to the canonical $\Z \oplus \Z$-action on $\R^2$, are in 1-1 relation with the shift-orbits of the resulting set of words $w(\ell)$.
Indeed, for this 1-1 relation it suffices to consider the positive half-words of any $w(\ell)$, so that 
it extends naturally to 
the lines $\ell$ that pass over any of the vertices.

We 
consider now more closely 
any of the ``troublesome'' lines $\ell_P$ that cross over 
a vertex 
$P$ of the square grid. 
To $\ell_P$ we associate two words $w_{\rm{above}}(\ell_P)$ and $w_{\rm{below}}(\ell_P)$ 
in $\{h, v\}^\Z$, which are read off from $\ell_P$ after isotopying it slightly in the neighborhood of $P$ so that it passes either above or below $P$.
From the above observed 1-1 relation between the $\Z \oplus \Z$-orbits of the lines $\ell$ and the shift-orbits of the corresponding words $w(\ell)$ we deduce that the words $w_{\rm{above}}(\ell_P)$ and $w_{\rm{below}}(\ell_P)$ do not belong to the same shift-orbit.

The set $X_s \subset \{h, v\}^\Z$ of all biinfinite words $w(\ell)$, including the above defined $w_{\rm{above}}(\ell_P)$ and $w_{\rm{below}}(\ell_P)$, for any line $\ell$ that passes through our distinguished base square $Q$, is a subshift - indeed, a well known Sturmian subshift.

\smallskip
\noindent
(2)
We now proceed by subdividing the top and bottom side of each square into segments of equal length through introducing a new vertex at the midpoint of any horizontal segment of the square grid. Any transition of a line $\ell$ through the left half of the subdivided horizontal square side will now be recorded by the letter $h_{\rm{left}}$, and any transition through the right half by $h_{\rm{right}}$, to give a new biinfinite word $w'(\ell) \in \{h_{\rm{left}}, h_{\rm{right}}, v\}^\Z$. The morphism $\sigma: \{h_{\rm{left}}, h_{\rm{right}}, v\}^\Z \to \{h, v\}^\Z$ defined by $h_{\rm{left}} \mapsto h, h_{\rm{right}} \mapsto h$ and  $v \mapsto v$ maps any $w'(\ell)$ to $w(\ell)$, and it will be 1-1, except for the new ``troublesome'' lines $\ell_R$ that pass through any of the new vertices $R$ in the middle of our original horizontal square grid intervals. For such lines we have as before 2 words $w'_{\rm{above}}(\ell_R)$ and $w'_{\rm{below}}(\ell_R)$, and both have the same image word $w(\ell_R)$.
Since $w'_{\rm{above}}(\ell_R)$ and $w'_{\rm{below}}(\ell_R)$ belong 
as above to distinct shift-orbits, the morphism $\sigma$ is not shift-orbit injective, and hence not a recognizable (see Proposition \ref{recognizable-e} (1)).

Clearly, this process can be iterated arbitrarily often, and every time the obtained morphism is shift-orbit injective except for two particular shift-orbits, which have the same image orbit.

\smallskip
\noindent
(3)
The above set-up of lines in a square grid of $\R^2$ admits a particularly convincing translation into an IET setting, since for any of the squares we can use the left hand and the bottom sides together as ``bottom intervals'', and the top side together with the right hand side as ``top intervals'', and the line segments of our foliation that are contained in the chosen square give canonically a classical IET system. If the chosen square agrees with the above picked base square $Q$, then the interval coding associated traditionally to the IET defines a subshift that agrees precisely  with the one given by the set of biinfinite words $w(\ell)$  (or similarly for $w'(\ell)$), which have been read off above from the intersections of the lines $\ell$ with the given square grid.
\end{rem}

After this ``appetizer'' we now give a detailed description of our construction in the subsequent 4 subsections. We assume a minimal familiarity with the basic terminology of Thurston's work on surfaces, such as ``pseudo-Anosov homeomorphism'', ``stable lamination'' or ``invariant train track''.

\subsection{The basic geometric quotient construction}
\label{sec:8.1}

${}^{}$

\smallskip

We will start by describing our basic geometric construction, using a pseudo-Anosov homeomorphism $h$ of a compact 
orientable 
surface $\Sigma$, and its expanding  invariant lamination $\Lambda^s$, which consists of uncountably many biinfinite geodesics (called ``leaves'') with respect to a fixed hyperbolic structure on $\Sigma$. [The family $\Lambda^s$ was called ``the stable lamination'' by Thurston, as he was looking at its behavior when lifted to 
the universal covering of $\Sigma$, identified with the hyperbolic plane $\Hy^2$, in the neighborhood of a $\partial\tilde h$-fixed point on $\partial \Hy^2$ (where $\tilde h$ is a lift of $h$ to $\Hy^2$ and $\partial\tilde h$ is the canonical extension of $\tilde h$ to $\partial \Hy^2$).] 

It is a standard procedure to translate such laminations (for instance by using an $h$-invariant train track neighborhood of $\Lambda^s$)
into a classical interval exchange setting, 
which in turn (assuming that $\Lambda^s$ is orientable and $\Sigma$ has at least one boundary component) allows a direct translation of $\Lambda^s$ into a 
subshift $X \subset\cal A^\Z$,
where 
$\cal A$ is given by the intervals in the IET.
Since both of these translations are well known (see for instance 
\cite{Del16}, \cite{Gadre}, \cite{Kap09}), 
we will restrict ourselves here only to a description of the geometry of $h$ and $\Lambda^s$.

For our purposes it is convenient to 
impose the following extra conditions:
\begin{enumerate}
\item[(H1)]
Assume that $\Sigma$ has $r \geq 2$ boundary components, 
which are all fixed by $h$. 
\item[(H2)]
Each complementary component of $\Lambda^s$ contains precisely one boundary component.

(Note that this assumption effects that there is a natural identification of $\pi_1 \Sigma$ with the free group $F(\cal A)$.)

\item[(H3)]
Each complementary component has at least 2 cusps, and each cusp is fixed by $h$. 
\end{enumerate}

We now pick
a particular complementary component $\Sigma_i \subset \Sigma$ of $\Lambda^s$, and assume that $\Sigma_i$ has precisely two cusps, and thus also precisely two boundary leaves $\ell_1$ and $\ell_2$, which 
(as do all boundary leaves of complementary components) 
will then both belong to $\Lambda^s$. We now pass to a quotient surface $\Sigma'$ by ``filling in'' the boundary component of $\Sigma$ that is contained in $\Sigma_i$, through identifying all points of the boundary curve in $\Sigma_i$ into a single point $P$ of $\Sigma'$. Then $h$ induces a pseudo-Anosov homeomorphism $h': \Sigma' \to \Sigma'$ with stable lamination $\Lambda'^s$, and there is a canonical quotient map $q: \Lambda^s \to \Lambda'^s$ that commutes with $h$ and $h'$ respectively. The map $q$ is 1-1 everywhere, except at points on the leaves $\ell_1$ and $\ell_2$, which are identified by $q$ to a single leaf $\ell' \in \Lambda'^s$. The leaf $\ell'$ is fixed and expanded by $h'$, and the sole $h'$-fixed point on $\ell'$ is precisely the above point $P$. This can be seen for example by the canonical passage from 
the stable lamination 
$\Lambda^s$ to the associated stable foliation $\cal F^s$ for $h$.

\begin{rem}
\label{minimal-to-minimal}
(1)
There is a remarkable feature here in that 
both, $\Lambda^s$ and $\Lambda'^s$ are minimal laminations (i.e. each leaf is dense), while the map $q$ induces on the leaf spaces of $\Lambda^s$ and $\Lambda'^s$ a map that is surjective, but not injective.

\smallskip
\noindent
(2)
This is translated (via the associated IETs as indicated above) into a subshift $X \subset \cal A^\Z$ that is mapped by a morphism $\sigma: \cal A^* \to \cal A'^*$ to a subshift $\sigma(X) =: X' \subset \cal A'^\Z$ 
(for $\cal A'^* \subset F(\cal A') = \pi_1 \Sigma'$, in complete analogy to $\cal A$ and $\Sigma$ in the above set-up). 
Here both, $X$ and $X'$, are minimal, while 
the map induced by $\sigma$ on $X$ is not shift-orbit injective, so that $\sigma$ is not recognizable in $X$.

\smallskip
\noindent
(3)
More precisely, since there is a natural 1-1 correspondence between the shift-orbits of $X$ and the leaves of $\Lambda^s$ (and similarly for $X'$ and $\Lambda'^s$), we observe that $\sigma$ maps precisely two shift-orbits of $X$ to a common image shift-orbit of $X'$, while everywhere else the induced map on shift-orbits is 1-1.
\end{rem}

\subsection{The ``inverse'' geometric quotient construction}
\label{sec:8.2}

${}^{}$

\smallskip

After having presented 
our basic geometric quotient construction, we will now describe the precise converse procedure: For this purpose we assume in this subsection that $\sigma_0, h_0, \Lambda_0^s, \cal A_0$ and $X_0$ are as $\Sigma, h, \Lambda^s, \cal A$ and $X$ in subsection \ref{sec:8.1} above, and that in particular the conditions (H1) - (H3) are satisfied, except that in (H1) we lower the assumption on the number $r$ of boundary components of $\Sigma_0$ to $r \geq 1$. We now select any non-boundary leaf $\ell_0$ of $\Lambda_0^s$ which is fixed by $h_0$: 
\begin{equation}
\label{fix-pt}
h_0(\ell_0) = \ell_0
\end{equation}
Since $\Lambda_0^s$ is expanded by $h_0$, it follows that there is precisely one fixed point $P = h_0(P) \in \ell$. We derive the surface $\Sigma_1$ from $\Sigma_0$ by puncturing a hole in $\Sigma_0$ at the point $P$, and observe from (\ref{fix-pt}) that $h_0$ induces a homeomorphism $h_1: \Sigma_1 \to \Sigma_1$. Again from considering the stable foliation $\cal F_0^s$ associated to $\Lambda_0^s$, we obtain the stable lamination $\Lambda_1^s \subset \Sigma_1$ for $h_1$ from $\Lambda_0^s$ by doubling the leaf $\ell_0$ into two leaves $\hat \ell_0$ and $\hat \ell'_0$, which are boundary leaves of a new complementary component $\Sigma'_1 \subset \Sigma_1$ that has no further boundary leaf. The component $\Sigma'_1$ contains a new boundary component 
of $\Sigma_1$ that runs around the puncture where formerly the point $P \in \Sigma_0$ was located. 

From this construction we obtain a 
quotient map $q_0: \Lambda_1^s \to \Lambda_0^s$ that 
satisfies
\begin{equation}
\label{eq8.2}
h_0 \circ q_0 \,\, = \,\, q_0 \circ h_1 \, ,
\end{equation}
and 
$q_0$ is 1-1 everywhere except on the leaves $\hat \ell_0$ and $\hat \ell'_0\,$, which are identified by $q_0$ to the single leaf $\ell_0 \in \Lambda_0^s\,$. We thus observe that the ``quotient procedure'' from $\Sigma_1, h_1$ and $\Lambda_1^s$ to $\Sigma_0, h_0$ and $\Lambda_0^s$ is precisely the same as described in subsection \ref{sec:8.1} when passing from $\Sigma, h$ and $\Lambda^s$ to $\Sigma', h'$ and $\Lambda'^s$.

\begin{rem}
\label{factor-map}
In the passage from $\Lambda_0^s$ to $\Lambda_1^s$, when translated into the IET language as in Remark \ref{minimal-to-minimal}, we observe that the IET for $\Lambda_1^s$ derives from the IET for $\Lambda_0^s$ by subdividing one of the intervals (namely the one onto which we choose to isotope $P$ along the leaf $\ell_0$). 
Hence the alphabet $\cal A_1$ for $\Lambda_1^s$ derives from  
$\cal A_0$
by doubling one of its letters, namely the one corresponding to the subdivided interval.

For the minimal subshift $X_1 \subset \cal A_1^\Z$ associated to $\Lambda_1$ and the morphism $\sigma_0: \cal A_1^* \to \cal A_0^*$ determined by the map $q_0$, which maps $X_1$ to 
$X_0$
and is non-recognizable in $X_1$, it follows that $\sigma_0$ is letter-to-letter, so that $X_0$ is actually a factor of $X_1\,$.
\end{rem}

\subsection{Iteration of the inverse  quotient construction}
\label{sec:8.3}

${}^{}$

\smallskip

We now look for a leaf $\ell_1 \in \Lambda_1^s$ with $h_1(\ell_1) = \ell_1\,$. As shown in the previous subsection, this is the only ingredient needed in order to repeat the above procedure to obtain a surface $\Sigma_2$, a pseudo-Anosov homeomorphism $h_2: \Sigma_2 \to \Sigma_2$ with stable lamination $\Lambda_2^s$, a map $q_1: \Lambda_2^s \to \Lambda_1^s$ and a morphism $\sigma_1: \cal A_2^* \to \cal A_1^*$ that is non-recognizable on the minimal subshift $X_2$ which satisfies $\sigma_1(X_2) = X_1\,$.

Hence, in order to be able to repeat this procedure infinitely often, with the purpose to get for any $n \geq 0$ a morphism $\sigma_{n}: \cal A_{n+1}^* \to \cal A_{n}^*$ that is non-recognizable on a minimal subshift $X_{n+1}$ with $\sigma_{n}(X_{n+1}) = X_{n}\,$, we just need for any $\Lambda_n^s$ a leaf $\ell_n \in \Lambda_n^s$ with $h_n(\ell_n) = \ell_n\, $. However, up to replacing $h_n$ by a power $h_n^{t(n)}$ for some suitable integer $t(n) \geq 1$, this is no problem: It is well known that any pseudo-Anosov map 
$h$ has infinitely many $h$-periodic 
leaves in its stable lamination.  We obtain
the following result, which is however only an intermediate step in our construction: In particular, the subshifts $X_n$ are {\em not} the intermediate level subshifts of the given directive sequence $\bvec \sigma$.
 
\begin{prop}
\label{infinite-non-recognize}
There exists 
a directive 
sequence $\bvec \sigma = (\sigma_n: \cal A_{n+1}^* \to \cal A_{n}^*)_{n \geq 0}$ 
and subshifts $X_n \subset \cal A_n^\Z\,$, such that for any $n \geq 0$ the following hold:
\begin{enumerate}
\item
$\sigma_n(X_{n+1}) = X_n\, $, and $\sigma_n$ is not recognizable 
in $X_{n+1}\,$.
\item
$\card(\cal A_{n+1}) = \card(\cal A_{n}) +1$
\item
$\sigma_n$ is letter-to-letter. In particular, $\sigma_n$ commutes with the shift operator, and $X_n$ is a factor of $X_{n+1}\,$. 
\item
$X_n$ is minimal, aperiodic and uniquely ergodic. 
\item
$X_n$ is 
substitutive 
(see Remark \ref{2.2d} (2)) 
for some primitive substitution $\tau_n: \cal A_n^* \to \cal A_n^*\,$.
\item
$\tau_n^{t(n)} \circ \sigma_n = \sigma_n \circ \tau_{n+1}$ for some integer $t(n) \geq 1$.
\end{enumerate}
\end{prop}

\begin{proof}
Properties (1), (2) and (3) have been derived in the construction described above. The substitution $\tau_n$ from (5) is the translation of the 
homeomorphism 
$h_n$ 
into the monoid setting through the canonical embedding $\cal A_n^* \subset F(\cal A_n) = \pi_1 \Sigma_n$.
The primitivity of $\tau_n$ is a direct consequence of the assumption ``pseudo-Anosov'' for $h$ and thus for all $h_n$. 
Property (4) is a direct consequence of (5), and (6) is 
the translation 
into the monoid setting 
of the commutativity relation $h_n^{t(n)} \circ q_{n} = q_{n} \circ h_{n+1}\, $, 
which is a consequence of the equality (\ref{eq8.2}) together with the above replacement of $h_n$ by $h_n^{t(n)}$.
\end{proof}

\subsection{Everywhere growing directive sequences that are not (eventually) recognizable}
\label{sec:8.4}

${}^{}$

\smallskip

The sequence $\bvec \sigma$ from Proposition \ref{infinite-non-recognize} is not everywhere growing; in fact, for any integers $m > n \geq 0$ the telescoped level map $\sigma_{[n, m)}$ is letter-to-letter. However, by choosing suitable ``diagonal'' or ``eventually horizontal'' paths through the infinite commutative diagram built from the above morphisms $\sigma_n$ (``vertical'') and $\tau_n$ (``horizontal'') we will derive below everywhere growing directive sequences with interesting properties.

Using the terminology from Proposition \ref{infinite-non-recognize}, we first define 
for each $n \geq 0$ the morphism 
$$\sigma'_n: = 
\tau_n^{t'(n)} \circ \sigma_n \,\,\, (\, = \sigma_n \circ 
\tau_{n+1}^{s(n)}\, ) ,$$
where we set $t'(n) := s(n) \, t(n)$ for some suitably chosen integer $s(n) \geq 1$ which ensures that the incidence matrix $M(\tau_n^{t'(n)})$ is positive. Such $s(n)$ exists because of property (5) of Proposition \ref{infinite-non-recognize}, and since $M(\sigma_n)$ has no zero-columns, it follows furthermore that
\begin{equation}
\label{eq7.2.5e}
\text{the incidence matrix $M(\sigma'_n)$ is positive, for any index $n \geq 0$.} 
\end{equation}

We now define a directive sequence $\bvec \sigma' = (\sigma'_n: \cal A_{n+1}^* \to \cal A_{n}^*)_{n \geq 0}$ with intermediate level subshifts called $X'_n$. Since $\tau_n(X_n) = X_n$ and $\sigma_n(X_{n+1}) = X_n$, we have 
\begin{equation}
\label{eq7.3e}
\sigma'_n(X_{n+1}) = X_n
\end{equation}
for any $n \geq 0$, so that 
from the minimality of $X_n$ we can deduce $X_n \subset X'_n$. In particular, we obtain from statement (1) of Proposition \ref{infinite-non-recognize} together with
Remark \ref{2.3e} 
that $\sigma'_n$ is not recognizable in $X_{n+1}$ and thus neither in $X'_{n+1}$.
From (\ref{eq7.2.5e}) we obtain directly 
(see Remark \ref{2.2d} (3)) 
that the sequence $\bvec \sigma'$ is everywhere growing.

\smallskip

Furthermore, we define for any integer $k \geq 0$ a directive sequence $\bvec \tau_{\! k} = (\tau'_n)
_{n \geq 0}$ through setting $\tau'_n := \tau_k$ for all $n \geq k$ and $\tau'_n = \sigma'_n$ if $0 \leq n \leq k-1$. We also specify the starting surface $\Sigma_0$ to be a punctured torus, so that one has $|\cal A_0| = 2$, and $X_0$ is Sturmian.
It follows that for any level $n \geq k$ the intermediate level $n$ subshift of $\bvec \tau_{\! k}$ is equal to the 
substitutive subshift $X_k$ defined by the substitution $\tau_k$ from statement (5) of Proposition \ref{infinite-non-recognize}, 
so that for every $0 \leq n \leq k-1$ we deduce from (\ref{eq7.3e}) that the level $n$ subshift is equal to $X_n$. The primitivity of $\tau_k$ implies in particular that the directive sequence $\bvec \tau_{\! k}$ is everywhere growing. Recall also 
that 
(as is true for all stationary sequences, 
see \cite{BPR21} and the references given there)
the truncated stationary sequence $\bvec \tau_{\! k} = (\tau'_n)_{n \geq k}$ is totally recognizable.

We obtain hence as immediate consequence of Proposition \ref{infinite-non-recognize} the following result; we observe that its parts (2) and (3) give directly the statements that have been rephrased in the Introduction and stated there as Proposition \ref{1.6}:

\begin{cor}
\label{2-interesting-sequences}
(1) 
The directive sequence $\bvec \sigma'$ is everywhere growing and satisfies 
the properties 
(1), (2), (4), (5) and (6) 
from Proposition \ref{infinite-non-recognize}, 
with $\sigma_n$ replaced by $\sigma'_n$.

\smallskip
\noindent
(2)
For any integer $k \geq 0$ there exists a directive sequence $\bvec \tau_{\! k}$,  with level alphabets $\cal A_n$ of 
size 
$\card(\cal A_n) = k +2$
for any level $n \geq k$, 
and $\card(\cal A_n) = n +2$ if $n \leq k$.
 
The sequence $\bvec \tau_{\! k}$ is everywhere growing and eventually recognizable: 
each of the first $k$ level morphisms on the bottom of $\bvec \tau_{\! k}$ is not recognizable in its corresponding level subshift, 
while all level morphisms of level $n \geq k$ are recognizable in their corresponding level subshift. Indeed, the sequence $\bvec \tau_{\! k}$ is stationary above level $k$.

\smallskip
\noindent
(3)
All intermediate level subshifts of the above directive sequences 
$\bvec \tau_{\! k}$ are minimal, 
uniquely ergodic 
and aperiodic.
In particular, the properties 
``recognizable'', ``shift-orbit injective'' (see Definition \ref{recognizable-e}) and ``recognizable for aperiodic points'' (see 
Remark \ref{added-late} (2))
are equivalent, 
for each level morphism 
in its corresponding intermediate level  subshift.
\qed
\end{cor}

\begin{rem}
\label{7.6e}
It turns out that property (3) of Corollary \ref{2-interesting-sequences} is also true for the directive sequence $\bvec \sigma'$. Indeed, 
from property (6) of Proposition \ref{infinite-non-recognize} and the well known North-South dynamics induced by any pseudo-Anosov homeomorphism of $\Sigma$ on the projectivized space of all measured laminations (= the boundary of  Teichm\"uller space for $\Sigma$) one can deduce that the inclusion $X_n \subset X'_n$ derived after (\ref{eq7.3e}) is actually an equality. However, laying out the details of these arguments would go beyond our self-imposed limits on the amount of Nielsen-Thurston theory imported into this section.
\end{rem}

\begin{rem}
\label{last}
Given 
any eventually recognizable everywhere growing directive sequence $\bvec \sigma = (\sigma_n)_{n \geq 0}$ of 
finite alphabet rank, 
one may ask 
whether there is an upper bound to the number level morphisms $\sigma_n$ 
which are not recognizable 
in their corresponding intermediate level subshift. 
This question 
has sparked some interest, see \cite{BSTY19} and \cite{DDMP}. It seems, however, that the examples given in part (2) of Corollary \ref{2-interesting-sequences} above contradict the bound claimed in Theorem 3.7 of \cite{DDMP}.
This could also effect the upper bound given in 
Corollary 1.5 of \cite{Espinoza} on the number of successive factor
maps, for a large class of subshifts.

In this context we 
also want to point to 
Example 7.5 of the very recent paper
\cite{3D+St}, where a family of directive sequences is presented that has the same properties as exhibited in Corollary \ref{2-interesting-sequences} (2) above for the sequences $\bvec \tau_k$. The examples from 
\cite{3D+St} are easier to describe, but fail to have the extra properties listed in part (3) of Corollary  \ref{2-interesting-sequences}.

Another construction of a similar kind (but closer our Corollary \ref{2-interesting-sequences} above) has been communicated to us by Basti\`an Espinoza \cite{Esp23} in the final stages of the revision of this paper.  
\end{rem}

\end{document}